\documentclass{article}
\usepackage[T1]{fontenc}
\usepackage{titlesec}
\usepackage[utf8]{inputenc}
\usepackage{amsmath}
\usepackage{amssymb}
\usepackage{amsfonts}
\usepackage{amsthm}
\usepackage{geometry}
\usepackage{xcolor}
\usepackage[hidelinks]{hyperref}
\newtheorem{lemma}{Lemma}

\newtheorem{theorem}{Theorem}
\numberwithin{equation}{section}
\title{PhD}
\author{}
\date{January 2021}

\begin{document}

\begin{center}
    \textbf{Finite Rank Perturbations of Heavy-Tailed Wigner Matrices}
\end{center}

\begin{center}
     Simona Diaconu\footnote{Department of Mathematics, Stanford University, sdiaconu@stanford.edu}
\end{center}

\begin{abstract}
    One-rank perturbations of Wigner matrices have been closely studied: let \(P=\frac{1}{\sqrt{n}}A+\theta vv^T\) with \(A=(a_{ij})_{1 \leq i,j \leq n} \in \mathbb{R}^{n \times n}\) symmetric, \((a_{ij})_{1 \leq i \leq j \leq n}\) i.i.d. with centered standard normal distributions, and \(\theta>0, v \in \mathbb{S}^{n-1}.\) It is well known \(\lambda_1(P),\) the largest eigenvalue of \(P,\) has a phase transition at \(\theta_0=1:\) when \(\theta \leq 1,\) \(\lambda_1(P) \xrightarrow[]{a.s.} 2,\) whereas for \(\theta> 1,\) \(\lambda_1(P) \xrightarrow[]{a.s.} \theta+\theta^{-1}.\) Under more general conditions, the limiting behavior of \(\lambda_1(P),\) appropriately normalized, has also been established: it is normal if \(||v||_{\infty}=o(1),\) or the convolution of the law of \(a_{11}\) and a Gaussian distribution if \(v\) is concentrated on one entry. These convergences require a finite fourth moment, and this paper considers situations violating this condition. For symmetric distributions \(a_{11},\) heavy-tailed with index \(\alpha \in (0,4),\) the fluctuations are shown to be universal and dependent on \(\theta\) but not on \(v,\) whereas a subfamily of the edge case \(\alpha=4\) displays features of both the light- and heavy-tailed regimes: two limiting laws emerge and depend on whether \(v\) is localized, each presenting a continuous phase transition at \(\theta_0=1, \theta_0 \in [1,\frac{128}{89}],\) respectively. These results build on our previous work~\cite{oldpaper} which analyzes the asymptotic behavior of \(\lambda_1(\frac{1}{\sqrt{n}}A)\) in the aforementioned subfamily.
\end{abstract}

\section{Introduction}\label{intro}
Large dimensional random matrices are used in numerous disciplines, including statistics, physics, genetics, and their spectral structure is central in most such applications (i.e., their eigenvalues and eigenvectors). The best understood models assume the building blocks are i.i.d. (observed features for sample covariance, and entries for Wigner matrices), with a vast amount of publications examining their limiting empirical distribution, the fluctuations of their edge eigenvalues, and the level spacings in the bulk of their spectrum (e.g., \cite{marchenkopastur}, \cite{johnstone}, \cite{tracywidom}). However, since the true distribution is unobservable, comprehending how robust these results are to changes in the underlying laws is indispensable. This has motivated the study of perturbations of random matrices, which come in various guises: multiplicative or additive, random or deterministic, finite or full rank. One of Péché's results (\cite{peche}) is seminal in a branch of this literature: suppose \(A=(a_{ij})_{1 \leq i,j \leq n} \in \mathbb{C}^{n \times n}\) is a Hermitian Wigner matrix with \((a_{ij})_{1 \leq i \leq j \leq n}\) independent, \(a_{ii}=N(0,c_{n,i}), a_{ij}=N(0,1), 1 \leq i<j \leq n,\) \(c_{n,i} \leq C.\) For the one-rank perturbation \(P(\theta)=\frac{1}{\sqrt{n}}A+\theta vv^*,\) with deterministic \(\theta>0, v \in \mathbb{C}^n, ||v||=1,\) the largest eigenvalue of \(P(\theta),\)
\begin{equation}\label{a.s.}
    \lambda_1(P(\theta)) \xrightarrow[]{a.s.} \begin{cases} 
    2, & \theta \leq 1 \\
    \theta+\theta^{-1}, & \theta \geq 1
    \end{cases}.
\end{equation}
Put differently, something occurs at \(\theta_0=1:\) below this threshold, the perturbation does not affect the asymptotic behavior of \(\lambda_1(P(\theta)),\) whereas above it, it does so through \(\theta\) (given the rotational invariance of Gaussian matrices, \(v,\) the eigenvector of the perturbation, should not affect such limits) in a way that allows identifying this parameter (\(x \to x+x^{-1}\) is injective on \([1,\infty)\)). Prior to this result, a similar phenomenon has been discovered for spiked covariance matrices. Let \(X=(x_{ij})_{1 \leq i \leq n, 1 \leq j \leq m}\) have i.i.d. columns (real- or complex-valued), and \(S=\frac{1}{m}XX^*\) be its sample covariance matrix. When \(\lim_{n \to \infty}{\frac{n}{m}}=\gamma \in (0,1)\) (part of the so-called random matrix theory regime), Bay and Yin~\cite{baiyin} showed \(\lambda_1(S) \xrightarrow[]{a.s.} (1+\sqrt{\gamma})^2\) (assuming solely \(\mathbb{E}[x_{11}]=0, \mathbb{E}[x^2_{11}]=1, \mathbb{E}[x^4_{11}]<\infty\)), after which Baik and Silverstein~\cite{baiksilv} found the almost sure limit of such covariance matrices with a feature of variance \(\theta \geq 1\) instead of one (their setting is more general than stated here: see \cite{baiksilv} for details):  
\[\lambda_1(S(\theta)) \xrightarrow[]{a.s.} \begin{cases} 
(1+\sqrt{\gamma})^2, & 1 \leq \theta \leq 1+\sqrt{\gamma} \\
\theta+\frac{\gamma \theta}{\theta-1}, & \theta \geq 1+\sqrt{\gamma}
\end{cases}.\]
Baik et al.~\cite{bbp}, expounded on this in the complex Gaussian case and discovered these two regimes yield fluctuations described by Tracy-Widom and normal distributions, respectively: this became known as the BBP-phase transition, after the initials of three authors of \cite{bbp}. 
\par
The publication~\cite{peche} generated plenty of interest, and subsequently several authors looked at finite rank perturbations of random matrices (mostly Wigner but also Wishart in the random matrix theory regime: e.g., \cite{maidaetal}). A finite fourth moment of the underlying distribution is a key assumption shared in all of these investigations, and universality (i.e., validity of these results for non-Gaussian scenarios) requires oftentimes a regularity such as the law satisfying a Poincaré or a log-Sobolev inequality. Although the latter set of conditions is likely suboptimal, the former is necessary inasmuch as the second moment relates to the convergence of the empirical spectral distribution (\cite{belinschietal}), while the fourth ties back to the almost sure limit of the largest eigenvalue being finite or infinite (\cite{baiyin2}).
\par
A classical result by Soshnikov~\cite{soshnikov} states the largest eigenvalues of Wigner matrices with heavy-tailed entries (i.e., their laws are regularly varying functions with index \(\alpha \in (0,4);\) notably, the fourth moment is infinite) converge to a Poisson point process. A transition occurs at \(\alpha=4:\) above it, the fluctuations are given by Tracy-Widom distributions (Lee and Yin~\cite{leeyin}), while at \(\alpha=4,\) the limiting laws can be continuous functions of Fréchet type (\cite{oldpaper}). Suppose \(A=(a_{ij})_{1 \leq i,j \leq n} \in \mathbb{R}^{n \times n}\) is symmetric, \((a_{ij})_{1 \leq i \leq j \leq n}\) i.i.d., \(a_{11} \overset{d}{=} -a_{11}, \mathbb{E}[a^2_{11}]=1, \lim_{x \to \infty}{x^4\mathbb{P}(|a_{11}|>x)}=c \in (0,\infty);\) then 
\begin{equation}\label{oldth}
    \lambda_1(\frac{1}{\sqrt{n}}A) \Rightarrow f(\zeta_c),
\end{equation}
where
\begin{equation}\label{fdef}
    f(x)=\begin{cases}
    2, & 0<x<1 \\
    x+\frac{1}{x}, & x \geq 1 
    \end{cases},
\end{equation}
and \(\zeta_c>0\) has a Fréchet distribution with shape and scale parameters \(4,(\frac{c}{2})^{1/4},\) respectively: for all \(x>0,\)
\begin{equation}\label{zetadef}
    \mathbb{P}(\zeta_c \leq x)=\exp(-\frac{cx^{-4}}{2}).
\end{equation}
Note (\ref{a.s.}) can be rewritten as \(\lambda_1(P(\theta)) \xrightarrow[]{a.s.} f(\theta);\) moreover, a function related to \(f\) appears in Tikhomirov and Youssef~\cite{tikhomirovyoussef} (\(\rho_n=\theta_n+\frac{np_n}{\theta_n}\)), in which the authors reveal a connection between the BBP phase-transition and sparse symmetric Wigner matrices (\(a_{11}=b_n\xi,\) \(\xi\) is a bounded random variable, and \(b_n\) is Bernoulli distributed of parameter \(p_n\) with \(\lim_{n \to \infty}{np_n}=\infty;\) theorem \(A\) states \(\frac{\lambda_k(A)}{\rho_n} \xrightarrow[]{p} 1\) for \(k \geq 1\) fixed, where \(\theta_n=\sqrt{\max_{1 \leq i \leq n}{(\sum_{1 \leq j \leq n}{a^2_{ij}}-np_n,np_n})}\)) and offer a heuristic for it (see discussion after the statement of theorem \(A\)).  
\par
Given the discrepancies between the light- and the heavy-tailed cases together with the large body of literature on finite rank perturbations of matrices with entries of the former type, a natural question is what happens with these perturbations when the underlying random variables belong to the latter category. Before analyzing this regime, consider two of the most commonly studied models in the light-tailed scenario:
\begin{equation}\label{perturb1}
    P_n=\frac{1}{\sqrt{n}}A+diag(\theta,0,0, \hspace{0.05cm} ... \hspace{0.05cm},0),
\end{equation}
\begin{equation}\label{perturb2}
    P_n=\frac{1}{\sqrt{n}}A+(\theta/n)_{1 \leq i,j \leq n},
\end{equation}
where \(A=(a_{ij})_{1 \leq i,j \leq n}\) is symmetric, \((a_{ij})_{1 \leq i \leq j \leq n}\) are i.i.d., \(\mathbb{E}[a_{11}]=0,\mathbb{E}[a^2_{11}]=1,\) and \(\theta>0\) is deterministic. Then \(\lambda_1(P_n)\) exhibits completely different behaviors in these two situations when \(\theta>1:\) Capitaine et al. show in \cite{capitaineetal} that for (\ref{perturb1}) its fluctuations depend heavily on \(\mu,\) the law of \(a_{11},\)
\[\frac{\theta^2}{\theta^2-1}\sqrt{n}(\lambda_1(P_n)-(\theta+\theta^{-1})) \Rightarrow \mu*N(0,v_{\theta}), \hspace{0.5cm} v_{\theta}=\frac{1}{2} \cdot \frac{\mathbb{E}[a_{11}^4]-3}{\theta^2}+\frac{1}{\theta^2-1},\]
as long as \(\mu\) is symmetric and satisfies a Poincaré inequality, whereas for (\ref{perturb2}), they are universal,
\[\sqrt{n}(\lambda_1(P_n)-(\theta+\theta^{-1})) \Rightarrow N(0,\sigma^2_{\theta}), \hspace{0.5cm} \sigma_\theta=\sqrt{1-\frac{1}{\theta^2}}\]
when \(\mu\) is symmetric with all its moments finite. The former result has been extended by Capitaine et al. in \cite{capitaineetal2}, where the authors prove that, up to a large extent, universal fluctuations of \(\lambda_1(M_n)\) are equivalent to the eigenvectors of the matrix perturbation being delocalized (i.e., their \(l_{\infty}\) norms converge to \(0\) as \(n \to \infty\)).
\par
Take now the analogues of (\ref{perturb1}) and (\ref{perturb2}) for the heavy-tailed regime. In virtue of the behavior in the i.i.d. case, the counterparts of (\ref{perturb1}) and (\ref{perturb2}) require a scaling dictated by the law of \(a_{11}.\) Suppose this distribution is regularly varying with index \(\alpha \in (0,4):\) i.e.,
\[\lim_{x \to \infty}{\frac{\mathbb{P}(|a_{11}| \geq x)}{x^{-\alpha} L(x)}}=1,\]
\(L:(0,\infty) \to (0,\infty)\) slowly varying (a measurable function with \(\lim_{x \to \infty}{\frac{L(ax)}{L(x)}}=1\) for any \(a>0\)). The restriction to \(\mathbb{N}\) of the regularly varying function \(b:(0,\infty) \to (0,\infty)\) with index \(2/\alpha,\) 
\[b(y)=\inf{\{x>0, \mathbb{P}(|a_{11}| \geq x) \leq 2y^{-2}\}}:=b_y\]
provides the correct normalization. 
A notable feature of \(b\) is
\[n^{2/\alpha-\epsilon} \leq b_n \leq n^{2/\alpha+\epsilon}\]
when \(\epsilon>0, n \geq n(\epsilon),\) a consequence of Karamata's representation theorem: 
\[L(x)=\exp{(\eta(x)+\int_{x}^{B}{\frac{\epsilon(t)}{t}dt})}\] 
for bounded, measurable functions \(\eta,\epsilon, \lim_{x \to \infty}{\eta(x)}=c(L) \in \mathbb{R}, \lim_{x \to \infty}{\epsilon(x)}=0,\) yielding
\(\lim_{x \to \infty}{\frac{L(x)}{x^\epsilon}}=0, \newline \lim_{x \to \infty}{x^{\epsilon}L(x)}=\infty\) for \(\epsilon>0.\) The finite rank perturbations above become
\[P_n(\theta,v_n)=b^{-1}_nA+\theta v_nv_n^T\]
for \(v_n=e_1, v_n=\frac{1}{\sqrt{n}}(e_1+e_2+...+e_n),\) respectively.

\subsection{Results}

Two families of distributions are considered: symmetric and regularly varying with index \(\alpha \in (0,4],\) with the case \(\newline \alpha=4\) partially covered. In the heavy-tailed regime \(\alpha \in (0,4)\) (i.e., infinite fourth moment), the fluctuations are universal, independent of the eigenvector \(v_n,\) and the limiting law always depends on \(\theta>0.\) 

\begin{theorem}\label{th2}
Suppose the distribution of \(a_{11}\) is symmetric and regularly varying with index \(\alpha \in (0,4).\) Let \(\theta>0, \newline v_n \in \mathbb{S}^{n-1}\) be deterministic, and
\[P_n(\theta,v_n)=b^{-1}_nA+\theta v_n^Tv_n.\]
Then
\begin{equation}\label{eq100}
    \lambda_{1}(P_n(\theta,v_n)) \Rightarrow \max{(\theta,E_{\alpha})},
\end{equation}
where \(E_{\alpha}>0, \mathbb{P}(E_{\alpha}<x)=\exp(-x^{-\alpha})\) for \(x>0.\)
\end{theorem}

The variational characterization of the largest eigenvalue of symmetric matrices \(P \in \mathbb{R}^{n \times n},\)
\begin{equation}\label{varchar}
    \lambda_1(P)=\sup_{u \in \mathbb{S}^{n-1}}{u^TPu},
\end{equation}
and its behavior in the unperturbed case 
\begin{equation}\label{convmax}
    \lambda_{1}(b^{-1}_nA) \Rightarrow E_{\alpha}
\end{equation}
offer some intuition for (\ref{eq100}). Soshnikov proves (\ref{convmax}) in \cite{soshnikov} by justifying \(\lambda_{1}(b^{-1}_nA) \approx b_n^{-1}\max_{1 \leq i \leq j \leq n}{|a_{ij}|},\) which can be easily seen to converge in law to \(E_{\alpha}\) (symmetry is not needed here). In particular, his proof implies the leading eigenvector of \(b^{-1}_nA\) is localized on two entries: if \(|a_{i_0j_0}|=\max_{1 \leq i \leq j \leq n}{|a_{ij}|},\) then with high probability \(i_0 \ne j_0,\) and \(u \in \mathbb{S}^{n-1}\) with \(u_{i_0}=|u_{j_0}|=\frac{1}{\sqrt{2}}, u_{j_0}a_{i_0j_0} \geq 0\) yields 
\[u^Tb^{-1}_nAu=b_n^{-1}|a_{i_0j_0}|+(2b_n)^{-1}(a_{i_0i_0}+a_{j_0j_0}) \approx b_n^{-1}\max_{1 \leq i \leq j \leq n}{|a_{ij}|}\]
as the maximum among the diagonal entries of \(A\) is negligible when scaled by \(b_n.\) At a high level, this entails the leading eigenvectors of the components of \(P_n(\theta,v_n)=b^{-1}_nA+\theta v_n^Tv_n\) are \(u,v_n,\) which with high probability are orthogonal because most of the entries of \(v_n\) are small (e.g., fewer than \(n^{1/2}\) have absolute value at least \(n^{-1/4}\)), making \(v_{n,i_0}, v_{n,j_0}\) likely to be among them. Furthermore, the same occurs if \(k\) leading eigenvectors of \(b^{-1}_nA\) are considered (\(k\) fixed) since these are also equally distributed on two entries with high probability, suggesting together with (\ref{varchar}) and (\ref{convmax}) that \(\lambda_1(P) \approx \max{(E_\alpha,\theta)}.\)
\par
In the edge case \(\alpha=4,\) \(v_n\) influences the asymptotic behavior of \(\lambda_1(P)\) primarily by its level of localization. If it is delocalized (i.e., \(||v_n||_{\infty}=o(1)\)), then the limiting distribution is universal. 

\begin{theorem}\label{th3}
Suppose \(a_{11} \overset{d}{=} -a_{11}, \lim_{x \to \infty}{x^4\mathbb{P}(|a_{11}|>x)}=c \in (0,\infty), \mathbb{E}[a^2_{11}]=1.\) Let \(\theta>0, v_n \in \mathbb{S}^{n-1}\) be deterministic, \(\lim_{n \to \infty}{||v_n||_{\infty}}=0,\) and 
\[P_n(\theta,v_n)=\frac{1}{\sqrt{n}} A+\theta v_n^Tv_n.\]
Then
\[\lambda_1(P_n(\theta,v_n)) \Rightarrow \max{(f(\zeta_c),F(\theta))},\]
where \(f,\zeta_c\) are defined by (\ref{fdef}), (\ref{zetadef}), respectively, and \(F:(0,\infty) \to (0,\infty),\)
\[F^2(\theta)=\limsup_{p \to \infty}{(s_1(\theta,p))^{1/p}},\]
\[s_1(\theta,p)=\sum_{s \geq 1, s \leq l \leq p-s/2}{\binom{2p-2l-1}{s-1}\theta^{2p-2l}\sum_{L_1, L_2, \hspace{0.05cm} ... \hspace{0.05cm},L_s \geq 1, L_1+...+L_s=l}{C_{L_1}C_{L_2}...C_{L_s}}}\]
for \(C_l=\frac{1}{l+1}\binom{2l}{l}, l \in \mathbb{N}.\)
\end{theorem}

\textit{Remark:} For \(G_1,G_2:(0,\infty) \to (0,\infty),\) 
\[G_1(y)=y \cdot \frac{2-2x}{2-3x}, \hspace{0.5cm} x=x(y) \in (0,2/3), \hspace{0.5cm} \frac{(2-3x)^3}{x(1-x)^2}=y^2,\]
\[G_2(y)=y \cdot \frac{2-2x}{2-7x/3}, \hspace{0.5cm} x=x(y) \in (0,6/7), \hspace{0.5cm} \frac{(6-7x)^{7/3}}{x^{1/3}(1-x)^2}=\frac{36y^2}{5^{2/3}},\]
it is shown that
\[\max{(2, G_2(\theta))} \leq F(\theta) \leq \max{(2,G_1(\theta))},\]
\[G_2(\theta) \leq 2 \Longleftrightarrow \theta \leq \frac{128}{89}, \hspace{0.5cm} G_2(\theta)<f(\theta), \hspace{0.5cm} \begin{cases}
    G_1(\theta)< f(\theta), & \theta < 1 \\
    G_1(\theta)> f(\theta), & \theta > 1
    \end{cases}.\]
(see subsection~\ref{3.1}, primarily Lemmas~\ref{lemma4}-\ref{lemma79}). In particular, \(F(\theta)=2\) for \(\theta \leq 1,\) and \(F(\theta)>2\) for \(\theta>\frac{128}{89},\) implying there is a phase transition at \(\theta_0=\inf{\{x, F(x)>2\}} \in [1,\frac{128}{89}],\) whereas in the localized scenarios covered below this occurs at \(\theta_0=1.\)

\begin{theorem}\label{th1}
Suppose \(a_{11} \overset{d}{=} -a_{11}, \lim_{x \to \infty}{x^4\mathbb{P}(|a_{11}|>x)}=c \in (0,\infty), \mathbb{E}[a^2_{11}]=1.\) Let \(\theta>0, k_n \in \mathbb{N}, v_n \in \mathbb{S}^{n-1}\) be deterministic with \(v_{n,(1)} \geq v_{n,(2)} \geq ... \geq v_{n,(n)}\) the ordered statistics of \((|v_{n,i}|)_{1 \leq i \leq n},\)
\[\liminf_{n \to \infty}{v_{n,(k_n)}}>0, \hspace{0.5cm} \lim_{n \to \infty}{\sum_{1 \leq i \leq k_n}{v^2_{n,(i)}}}=1,\]
\[P_n(\theta,v_n)=\frac{1}{\sqrt{n}}A+\theta v_n^Tv_n.\]
Then
\[\lambda_1(P_n(\theta,v_n)) \Rightarrow \max{(f(\theta),f(\zeta_c))}.\]
\end{theorem}

\par
Some observations on these last two results are in order. Since \(\lambda_1(\frac{1}{\sqrt{n}}A) \Rightarrow f(\zeta_c),\)
where similarly to the heavy-tailed case, \(\zeta_c\) gives the asymptotic law of \(max(A):=\frac{1}{\sqrt{n}}\max_{1 \leq i \leq j \leq n}{|a_{ij}|},\) the structure of the leading eigenvector of \(\frac{1}{\sqrt{n}}A\) is not as transparent as for \(\alpha \in (0,4).\) At a heuristic level, it cannot be localized because one can think of the case \(max(A) \leq 1\) as falling into the light-tailed regime, in which it is known that, under certain conditions on the underlying distribution, the leading eigenvectors are delocalized (e.g., theorem \(5.1\) in \cite{erdosetal}), whereas when \(max(A)>1,\) some localization should occur due to \(max(A)+(max(A))^{-1} \approx max(A)\) for \(max(A)\) very large, although definitely not as clear-cut as for \(\alpha \in (0,4)\). In virtue of these observations, expecting the localization of \(v_n\) to influence the asymptotic behavior seems reasonable. When \(||v_n||=o(1),\) there is a global contribution from the perturbation encompassed by \(F,\) while for localized \(v_n,\) \(\theta\) influences the limiting law through \(f,\) the function appearing in the unperturbed case. Moreover, it will become apparent in the coming section that the symmetry of \(a_{11}\) is a vestige of the employed methodology rather than a necessary condition.
\par
All three convergences rely on finding asymptotically equal upper and lower bounds. Obtaining the latter is relatively straightforward by the variational characterization of eigenvalues of symmetric matrices \(M,\) whereas for the former,
\[\lambda_1(M) \leq (tr(M^{2p}))^{1/(2p)}\] 
for any \(p \in \mathbb{N}.\) Specifically, upper bounds for \(\mathbb{E}[tr(M^{2p})]\) and Markov's inequality provide such inequalities in probability. For Theorem~\ref{th2}, these two bounds are tight and render the conclusion once a proxy \(\tilde{P}\) for \(P\) is chosen. When \(\alpha=4,\) these quantities are different, the larger being optimal. Hence, a new tool is imperative to produce a tight lower bound, and the proofs employ a modified method of moments, whose original version was introduced by Sinai and Soshnikov in \cite{sinaisosh}. This technique, previously used in \cite{oldpaper}, builds on the following observation: suppose \(S,Q \in \mathbb{R}^{n \times n}\) are symmetric with \(S\) very sparse (say, \(o(n)\) nonzero entries). Then \(Q\) and \(S+Q\) share, roughly speaking, most of their eigenvalues and thus 
\begin{equation}\label{1}
    tr((S+Q)^{p})-tr(Q^p)
\end{equation}
is determined by the extrema of their eigenspectra as their bulks cancel out one another. Because \(tr(M^{p})\) encapsulates the size of \(||M||\) for \(p\) even, and of \(\lambda_1(M)\) relative to \(||M||\) for \(p\) odd, in the former case, if \(||Q||< ||S+Q||,\) then (\ref{1}) is of order \(||S+Q||^p-||Q||^p \approx ||S+Q||^p\) for \(p\) sufficiently large, whereas in the latter, if \(\lambda_1(S+Q)<||S+Q||,||Q||< ||S+Q||,\) then (\ref{1}) is approximately \(-||S+Q||^p-||Q||^p \approx -||S+Q||^p.\) After finding such matrices \(S,Q\) with \(||S+Q-P||=o(1),\) a two-step approach can be used to obtain the asymptotic behavior of \(\lambda_1(P):\) \(p\) even gives the size of \(||P||\) (as long as the variance of (\ref{1}) is small), and subsequently \(p\) odd implies \(\lambda_1(P)\) cannot be smaller than \(||P||\) (otherwise, (\ref{1}) is large in absolute value and negative, which occurs with small probability as its mean will be nonnegative and its variance small). In particular, Theorems~\ref{th2}-\ref{th1} hold for \(||P_n(\theta,v_n)||\) as well. 
\par
For simplicity, drop the dependency on \(n\) and \(\theta:\) \(v=v_n, P=P_n(\theta,v_n),\) and by convention, any constant is positive unless otherwise stated. Section \ref{sect0} presents in more detail the proof strategy behind these results, starting with summary of the proof of theorem \(1\) in \cite{oldpaper}, upon which all the results of this paper are built, and continuing with the changes the perturbed matrices considered impose on this approach. Sections \ref{sect1}-\ref{sect2} contain the justifications of Theorems \ref{th2}-\ref{th1}, respectively.

\section{Trace Moments and Largest Eigenvalues}\label{sect0}

The convergences stated in Theorems~\ref{th2}-\ref{th1} rely on Slutsky's lemma: squeezing \(\lambda_1(P)\) between two random quantities whose laws are asymptotically the same (the crucial random variable is \(\max_{1 \leq i,j \leq n}{|a_{ij}|},\) scaled by \(b_n\) or \(\sqrt{n},\) with limiting distribution \(E_\alpha, \zeta_c,\) respectively), and fine bounds of conditional trace differences such as (\ref{1}) are the key towards obtaining these inequalities. Subsection~\ref{0.1} summarizes the justification of theorem \(1\) in \cite{oldpaper}, and subsection \ref{0.2} presents the changes the current situation requires.

\subsection{Unperturbed Wigner Matrices}\label{0.1}

The proof of \(\lambda_1(\frac{1}{\sqrt{n}}A) \Rightarrow f(\zeta_c),\) under the assumptions of Theorem~\ref{th3}, in \cite{oldpaper} can be outlined as:
\par
\(1.\) For \(\kappa>0\) fixed, decompose \(A\) into four matrices, containing small, medium, big, and very big entries, respectively:
\[\frac{1}{\sqrt{n}} A=A_s+A_m+A_{b,\kappa}+A_{B,\kappa},\]
where for any \(\lim_{n \to \infty}{m_n}=\infty,\) with high probability \(||A_m||=o(1), ||A_{b,\kappa}|| \leq \kappa,\) and \(A_{B,\kappa}\) has at most \(m=m_n\) nonzero entries. 
\par
\(2.\) Justify
\begin{equation}\label{condtrace}
    \mathbb{E}_{*}[tr((A_s+A_{B,\kappa})^{2p})-tr(A_s^{2p})] \leq 2m \cdot s(M,p)+O(n^{-\delta}(M+1)^{2p}C(p)),
\end{equation}
\begin{equation}\label{condvar}
    Var_{*}[tr((A_s+A_{B,\kappa})^{2p})-tr(A_s^{2p})] \leq O(n^{-\delta}(M+1)^pC(p)),
\end{equation}
for \(*\) denoting conditioning on \(max(A):=\frac{1}{\sqrt{n}}\max_{1 \leq i \leq j \leq n}{|a_{ij}|} \leq M\) and an event whose probability tends to one as \(n \to \infty,\) constants \(M,\delta>0,\) an increasing function \(C:\mathbb{N} \to \mathbb{N},\) a family of polynomials \(s_d(X):=s(X,d)\) of degree \(2d, d \in \mathbb{N},\) whose coefficients have a combinatorial definition and with \(\lim_{d \to \infty}{(s(x,d))^{1/{d}}}=f^2(x)\) for \(x>0.\)
\par
\(3.\) Prove an inequality in spirit reverse to (\ref{condtrace}), and an analogue of (\ref{condvar}):
\begin{equation}\label{condexp2}
    \mathbb{E}_{**}[tr((A_s+A_{B,\kappa})^{2p+1})-tr(A_s^{2p+1})]=0,
\end{equation}
\begin{equation}\label{condvar2}
    Var_{**}[tr((A_s+A_{B,\kappa})^{2p+1})-tr(A_s^{2p+1})] \leq O(n^{-\delta}(M+1)^{2p+1}C(p)),
\end{equation}
where \(**\) denotes (another) conditioning on \(max(A):=\frac{1}{\sqrt{n}}\max_{1 \leq i \leq j \leq n}{|a_{ij}|} \leq M,\) and an event whose probability tends to one as \(n \to \infty.\)
\par
These ingredients give the desired result: \(1.\) yields \(||A_s+A_{B,\kappa}||\) as a proxy for \(||A||\) (let \(\kappa \to 0\)); since \(A_{B,\kappa}\) is by design very sparse, \(tr((A_s+A_{B,\kappa})^{2p})-tr(A_s^{2p}) \approx ||A_s+A_{B,\kappa}||^{2p}\) from Lemma~\ref{lemma7old} and \(||A_s||\) being asymptotically not larger than \(||A_s+A_{B,\kappa}||\) (for \(\epsilon>0,\) the former is at most \(2+\epsilon\) by computing \(\mathbb{E}[tr(A_s^{2p})],\) the latter at least \(2-\epsilon\) because it is close to the operator norm of \(A,\) whose empirical spectral distribution converges to the semicircle law), implying together with \(2.\) that \(||A|| \Rightarrow f(\zeta_c)\) (the trace difference in (\ref{condtrace}) is lower bounded by roughly \(s(max(A),p)\) by arguing in the same vein as for the upper bound).

\begin{lemma}[{Lemma \(7,\) \cite{oldpaper}}]\label{lemma7old}
Suppose \(S, Q \in \mathbb{R}^{n \times n}\) are symmetric matrices, \(\lambda_{2m+1}(Q)=0\) for an integer \(m \in [1,\frac{n}{6}-1].\) Then for \(p \in \mathbb{N},\)
\[||S+Q||^{2p}-7m \cdot ||S||^{2p} \leq tr((S+Q)^{2p})-tr(S^{2p}) \leq 4m \cdot ||S+Q||^{2p}.\]
\end{lemma}

If \(\lambda_1(A_s+A_{B,\kappa})<||A_s+A_{B,\kappa}||-\epsilon,\) then \(tr((A_s+A_{B,\kappa})^{2p+1})-tr(A_s^{2p+1}) \approx -||A_s+A_{B,\kappa}||^{2p+1}\) from Lemma~\ref{lemma9old} (\(\lambda_n(A_s+A_{B,\kappa})=-||A_s+A_{B,\kappa}||\)), which is large in absolute value and negative, whereas \(3.\) entails \(tr((A_s+A_{B,\kappa})^{2p+1})-tr(A_s^{2p+1}),\) modulo an event of small probability, concentrates around \(0\) for \(p \leq c(n).\)

\begin{lemma}[{Lemma \(9,\) \cite{oldpaper}}]\label{lemma9old}
Suppose \(S, Q \in \mathbb{R}^{n \times n}\) are symmetric matrices, \(\lambda_{2m+1}(Q)=0\) for an integer \(m \in [1,\frac{n}{4}-1].\) Then for \(p \in \mathbb{N},\)
\[tr((S+Q)^{2p+1})-tr(S^{2p+1}) \leq 2m \cdot (\lambda_1(S+Q))^{2p+1}+(\lambda_n(S+Q))^{2p+1}+3m \cdot ||S||^{2p+1}.\]
\end{lemma}

\par
The conditional expectation (\ref{condtrace}) is computed using a modified version of the combinatorial technique pioneered by Sinai and Soshnikov in \cite{sinaisosh} (employed in Sinai and Soshnikov~\cite{sinaisosh2}, Soshnikov~\cite{sohnikov2}, Auffinger et al.~\cite{auffinger}, etc.). This counting device has been initially used to bound expectations of traces of large powers of random matrices, which in the heavy-tailed regime can be stated as: 

\vspace{0.2cm}
Suppose \(p \in \mathbb{N},\) \(B=(b_{ij})_{1 \leq i,j \leq n} \in \mathbb{R}^{n \times n}\) is a symmetric random matrix for which \((b_{ij})_{1 \leq i \leq j \leq n}\) are i.i.d., with symmetric distributions, \(\mathbb{E}[b_{11}^2] \leq 1, \mathbb{E}[b_{11}^{2l}] \leq L(n)n^{\delta(2l-4)}, 2 \leq l \leq p, \delta>0, L:\mathbb{N} \to [1,\infty), L(n) \leq n^{2\delta}.\) Then
\begin{equation}\label{trace}
    \mathbb{E}[tr(B^{2p})] \leq L(n)2^{2p}p!\sum_{(n_1,\hspace{0.05cm} ...\hspace{0.05cm} ,n_p)}{n^{1+\sum_{1 \leq k \leq p}{n_k}+2\delta\sum_{k \geq 2}{kn_k}}\prod_{1 \leq k \leq p}{\frac{1}{(k!)^{n_k}n_k!}}\prod_{2 \leq k \leq p}{(2k)^{kn_k}}}
\end{equation}
where \(n_1,\hspace{0.05cm} ...\hspace{0.05cm} ,n_p \in \mathbb{Z}_{\geq 0}, \sum_{1 \leq k \leq p}{kn_k}=p.\) This entails for \(p=o(n^{1/2-2\delta}),\)
\begin{equation}\label{traceineq}
    \mathbb{E}[tr(B^{2p})] \leq 2^{2p}L(n)n^{p+1}e^8.
\end{equation}
\vspace{0.2cm}
\par
The key towards (\ref{trace}) is a change of summation in
\[\mathbb{E}[tr(B^{2p})]=\sum_{(i_0, i_1, \hspace{0.05cm} ... \hspace{0.05cm}, i_{2p-1})}{\mathbb{E}[b_{i_0i_1}b_{i_1i_2}...b_{i_{2p-1}i_{0}}]}:=\sum_{\mathbf{i}=(i_0, i_1, \hspace{0.05cm} ... \hspace{0.05cm}, i_{2p-1},i_0)}{\mathbb{E}[b_{\mathbf{i}}]},\]
from cycles \(\mathbf{i}:=(i_0,i_1, \hspace{0.05cm} ...\hspace{0.05cm} ,i_{2p-1},i_0)\) to tuples \((n_1,\hspace{0.05cm} ...\hspace{0.05cm} ,n_p),\sum_{1 \leq k \leq p}{kn_k}=p.\) 
Before sketching how this is accomplished (subsection \(2.1\) in \cite{oldpaper} contains their full justification), let us introduce some needed terminology and notation. 
\par
Interpret \(\mathbf{i}=(i_0,i_1, \hspace{0.05cm} ...\hspace{0.05cm} ,i_{q-1},i_0)\) as a directed cycle of length \(q,\) with vertices among \(\{1,2, \hspace{0.05cm} ... \hspace{0.05cm}, n\};\) call \((i_{k-1},i_k)\) (\(i_{k-1}\)) its \(k^{th}\) edge (vertex) for \(1 \leq k \leq q,\) where \(i_{q}:=i_0;\) for \(u,v \in \{1,2, \hspace{0.05cm} ... \hspace{0.05cm}, n\},\) \((u,v)\) denotes a directed edge from \(u\) to \(v,\) whereas \(uv=vu\) is undirected (the former are the building blocks of the cycles underlying the trace in (\ref{trace}), while the latter determine their expectations). Call \(\mathbf{i}\) an \textit{even cycle} if each undirected edge appears an even number of times in it, an edge \((i_k,i_{k+1})\) and its right endpoint \(i_{k+1}\) \textit{marked} if an even number of copies of \(i_ki_{k+1}\) precedes it: i.e., if \(\{t \in \mathbb{Z}: 0 \leq t \leq k-1, i_ti_{t+1}=i_ki_{k+1}\}\) has even size; lastly, pair each unmarked edge \((i_k,i_{k+1})\) with its last marked copy (couple it with \((i_{t'},i_{t'+1}),\) where \(t'=\max{\{t \in \mathbb{Z}: 0 \leq t \leq k-1, i_ti_{t+1}=i_ki_{k+1}\}}\)). The analysis of these directed graphs, and consequently of (\ref{trace}), relies on this pairing, and on \(\mathbb{E}[b_{\mathbf{i}}]=0\) unless \(\mathbf{i}\) is even (due to symmetry).
\par
Given the introduced definitions, even cycles are the sole contributors to (\ref{trace}), and each such \(\mathbf{i}\) has \(p\) marked edges, with any vertex \(j \in \{1,2, \hspace{0.05cm} ... \hspace{0.05cm}, n\}\) in it, apart perhaps from \(i_0,\) marked at least once (the first edge of \(\mathbf{i}\) containing \(j\) is of the form \((i,j)\) since \(i_0 \ne j,\) and no earlier edge is adjacent to \(j\)). For \(0 \leq k \leq p,\) denote by \(N_{\mathbf{i}}(k)\) the set of \(j \in \{1,2, \hspace{0.05cm} ... \hspace{0.05cm}, n\}\) marked exactly \(k\) times in \(\mathbf{i}\) with \(n_k:=|N_{\mathbf{i}}(k)|.\) Then
\begin{equation}\label{tuplecond}
    \sum_{0 \leq k \leq p}{n_k}=n, \hspace{0.2cm} \sum_{1 \leq k \leq p}{kn_k}=p,
\end{equation}
providing a change of summation for the left-hand side of (\ref{trace}). With this construction under the belt, the remaining task
is obtaining upper bounds for the number of such cycles taken to a given tuple (steps \(1-4\)) and their individual contributions (step \(5\)). In what follows, \((n_1,n_2, \hspace{0.05cm} ... \hspace{0.05cm}, n_p)\) stays fixed, and \(\mathbf{i}\) is an even cycle mapped to it by the aforesaid procedure.
\par
\underline{Step \(1.\)} Map \(\mathbf{i}\) to a Dyck path \((s_1,s_2, \hspace{0.05cm} ... \hspace{0.05cm} ,s_{2p}),\) where \(s_k=+1\) if \((i_{k-1},i_k)\) is marked, and \(s_k=-1\) if \((i_{k-1},i_k)\) is unmarked.
\par
\underline{Step \(2.\)} Once the positions of the marked edges in \(\mathbf{i}\) are chosen (i.e., a Dyck path), establish the order of their marked vertices (out of \(p\) pairwise disjoint sets \((N_{\mathbf{i}}(k))_{1 \leq k \leq p}\) with \(|N_{\mathbf{i}}(k))|=n_k\)). 
\par
\underline{Step \(3.\)} Select the vertices of \(\mathbf{i},\)
\[V(\mathbf{i}):=\cup_{0 \leq k \leq 2p-1}{\{i_k\}},\] 
one at a time from \(\{1,2, \hspace{0.05cm} ... \hspace{0.05cm} ,n\}\), by reading its edges in order, starting at \((i_0,i_1)\) (\(|V(\mathbf{i})| \leq 1+\sum_{1 \leq k \leq p}{n_k}\) in virtue of an earlier observation).  
\par
\underline{Step \(4.\)} Choose the remaining vertices of \(\mathbf{i}\) from \(V(\mathbf{i}),\) by traversing anew its edges in increasing order (step \(3\) only established the first appearance of the elements of \(V(\mathbf{i})\) in \(\mathbf{i}\)).
\par
\underline{Step \(5.\)} Bound the expectation generated by such cycles \(\mathbf{i}.\) 

\vspace{0.3cm}
An altered version of this technique allows computing
\begin{equation}\label{tracediff}
    \mathbb{E}_{(*)}[tr((A_s+A_{B,\kappa})^{p})-tr(A_s^{p})],
\end{equation}
where \((*)\) denotes a conditioning of interest. Since \(tr((A_s+A_{B,\kappa})^{p})\) can be written as a sum over cycles of length \(p,\) whose edge contributions come either from the corresponding entry in \(A_s\) or \(A_{B,\kappa}\) (in such situations, say \(a_{ij}\) \textit{belongs to} \(A_s, A_{B,\kappa},\) respectively), the trace difference is a summation over cycles which contain at least one entry belonging to \(A_{B,\kappa}.\) An important family in this regard is \((\mathcal{C}(l))_{l \in \mathbb{N}},\) where \(\mathcal{C}(l)\) denotes the set of pairwise non-isomorphic even cycles of length \(2l,\) with \(n_1=l\) and the first vertex unmarked (call two cycles \(\mathbf{i},\mathbf{j}\) of length \(q\) \textit{isomorphic} if \(i_s=i_t \Longleftrightarrow j_s=j_t\) for all \(0 \leq s,t \leq q\)), and \(b_{l,t}\) the number of vertices \(v \in V(\mathbf{i}), \mathbf{i} \in \mathcal{C}(l)\) of multiplicity \(t:\) i.e., \(|\{0 \leq j \leq 2l, i_j=v\}|=t.\) A close look at the proof of (\ref{trace}) reveals the elements of \(\mathcal{C}(l)\) dominate the trace, the rest producing contributions of smaller order (this occurs primarily because for \(\mathbf{i} \not \in \mathcal{C}(l),\) \(|V(\mathbf{i})|<l+1;\) if the moment bounds in (\ref{trace}) were independent of \(n,\) then these vertices would generate the leading factor in such products: \(n(n-1)...(n-|V(\mathbf{i})|+1) \approx n^{|V(\mathbf{i})|},\) while the rest would be controlled by \(l;\) something similar to this occurs also under the assumptions behind (\ref{trace})). In (\ref{tracediff}), due to the large nonzero entries of \(A_{B,\kappa},\) there are more cycles making a non-negligible contribution than \(\mathcal{C}(p):\) nevertheless, they can still be described using this family (see Lemma~\ref{lemma3old}), and the sequence \((b_{l,t})_{l,t \in \mathbb{N}}\) is crucial when counting them.
\par
The strategy for (\ref{tracediff}) relies anew on a change of summation, which now has fundamentally one additional parameter: the edges belonging to \(A_{B,\kappa}.\) For \(\mathbf{i},\) let \(\mathbf{i}',\mathbf{i}''\) be the strings of \(2p\) elements such that for \(0 \leq t \leq 2p-1,\) if \(a_{i_t i_{t+1}}\) belongs to \(A_s,\) then \(\mathbf{i}'_t=(i_t,i_{t+1}),\mathbf{i}''_t=\emptyset;\) else,  \(\mathbf{i}'_t=\emptyset, \mathbf{i}''_t=(i_t,i_{t+1}),\) and adopt \(\mathbf{i}=(\mathbf{i}',\mathbf{i}'')\) as a shorthand for this decomposition. Put differently, \(\mathbf{i}',\mathbf{i}''\) record the edges of \(\mathbf{i}\) belonging to \(A_s\) and \(A_{B,\kappa},\) respectively: by ignoring the empty set entries in these sequences, they can be naturally seen as subgraphs of \(\mathbf{i},\) an interpretation implicitly assumed henceforth. Steps \(1'-5'\) below consider the contributions of cycles \(\mathbf{i}\) for \(\mathbf{i}''\) fixed, while step \(6'\) sums them over all such directed graphs.
\par
Regarding the conditionings of interest, \(*\) fixes the positions of the nonzero entries of \(A_{B,\kappa},\) and \(**\) additionally sets their values. By independence, the conditional moments of the entries belonging to \(A_s\) can be replaced by their unconditional counterparts at the cost of a factor \(c(\kappa,c):\) for any nonnegative integers \((p_{ij})_{1 \leq i \leq j \leq n},\)
\[\mathbb{E}_*[\prod_{1 \leq i \leq j \leq n}{a^{p_{ij}}_{ij}}]=\prod_{1 \leq i \leq j \leq n, (i,j) \in S}{\mathbb{E}[a^{p_{ij}}_{ij} \hspace{0.05cm}| \hspace{0.05cm} \kappa \sqrt{n}<|a_{ij}| \leq M \sqrt{n}]} \cdot \prod_{1 \leq i \leq j \leq n, (i,j) \not \in S}{\mathbb{E}[a^{p_{ij}}_{ij}\chi_{|a_{ij}| \leq n^{\delta}} \hspace{0.05cm}| \hspace{0.05cm} |a_{ij}| \leq \kappa \sqrt{n}]}.\]
By symmetry, if there is some \(p_{ij}\) odd, then the expectation is zero; else,
\begin{equation}\label{sub}
    \mathbb{E}_*[\prod_{1 \leq i \leq j \leq n}{a^{p_{ij}}_{ij}}] \leq c(\kappa,c)M^{\sum_{(i,j) \in S}{p_{ij}}} \cdot  \prod_{1 \leq i \leq j \leq n, (i,j) \not \in S}{\mathbb{E}[a^{p_{ij}}_{ij}\chi_{|a_{ij}| \leq n^{\delta}}]}.
\end{equation}
since \(\mathbb{P}(|a_{11}| \leq \kappa \sqrt{n}) \geq 1-2c(\kappa \sqrt{n})^{-4}\) and \(|S| \leq m.\) This and Lemma \(8\) in \cite{oldpaper} (if a cycle is not even and contains at least an edge belonging to \(A_s,\) then it has one of odd multiplicity) imply only even cycles matter, an observation essential in the analysis to come.
\par
Proceed now with the first five steps: since \(\mathbf{i}''\) is fixed at this stage (denote its length by \(2l\)), the summation is over \(\mathbf{i}'\) with \(\mathbf{i}=(\mathbf{i}',\mathbf{i}''),\) and as in the classical case, the goal is switching from \(\mathbf{i}'\) to a tuple \((n'_1,n'_2,\hspace{0.05cm} ...\hspace{0.05cm} ,n'_{p-l}), \sum_{1 \leq k \leq p-l}{kn'_k}=p-l.\) For an even cycle \(\mathbf{i},\) let \(N'_{\mathbf{i}}(k)\) be the set of vertices of \(\mathbf{i}\) appearing as right endpoints of marked edges of \(\mathbf{i}'\) exactly \(k\) times, and \(n'_k:=|N'_{\mathbf{i}}(k)|\) for \(1 \leq k \leq p-l\) (since \(\mathbf{i}'\) and \(\mathbf{i}''\) share no undirected edge, marking them either separately or jointly in \(\mathbf{i}\) leads to the same configuration). In what follows, \(\mathbf{i}=(\mathbf{i}',\mathbf{i}'')\) is even with \((n'_1,n'_2,\hspace{0.05cm} ...\hspace{0.05cm} ,n'_{p-l})\) fixed, and \(1 \leq l \leq p-1.\) Although steps \(1-5\) do not generally hold when \((n_1,n_2,\hspace{0.05cm} ...\hspace{0.05cm} ,n_p)\) is replaced by \((n'_1,n'_2,\hspace{0.05cm} ...\hspace{0.05cm} ,n'_{p-l}),\) they can be modified and still yield useful bounds.
\par
\underline{Step \(1'.\)} Map the edges of \(\mathbf{i}'\) to a Dyck path of length \(2p-2l.\)
\par
\underline{Step \(2'.\)} Establish the order of the marked vertices in \(\mathbf{i}'\) (out of \(p-l\) pairwise disjoint sets \((N'_{\mathbf{i}}(k))_{1 \leq k \leq p-l}\) with \(|N'_{\mathbf{i}}(k))|=n'_k.\))
\par
\underline{Step \(3'.\)} Select the vertices of \(\mathbf{i}',\) \(V(\mathbf{i}'),\) one at a time by reading its edges in order: each vertex of \(\mathbf{i}'\) is \(i_0,\) marked at least once, or adjacent to an edge in \(\mathbf{i}''.\)
\par
\underline{Step \(4'.\)} Choose the remaining vertices of \(\mathbf{i}'\) among the ones selected in step \(3'.\) 
\par
\underline{Step \(5'.\)} Bound the expectation generated by \(\mathbf{i'}:\) 
\[\mathbb{E}[a_{\mathbf{i}'}] \leq L(n)^{(p-l)/2}n^{2\delta\sum_{k \geq 2}{kn'_k}}.\]
Furthermore, an overall saving of a power of \(n\) is possible unless \(\mathbf{i}\) has a very special form:

\begin{lemma}[{Lemma \(3,\) \cite{oldpaper}}]\label{lemma3old}
For an even cycle \(\mathbf{i}\) with \(l<p,\) at least one of the following occurs:
\par
\((I)\) a factor of \(n^{2\delta}\) can be saved in step \(5':\)
\[\mathbb{E}[a_{\mathbf{i}'}] \leq L(n)^{(p-l)/2}n^{2\delta(\sum_{k \geq 2}{kn'_k}-1)},\]
\par
\((II)\) a factor of \(n/(2m)\) can be dropped in step \(3',\)
\par
\((III)\) \(n'_1=p-l,\) \(i_0\) is unmarked in \(\mathbf{i}',\) \(\mathbf{i''}\) contains a unique undirected edge \(vw\) with \(v \in N'_{\mathbf{i}}(1) \cup \{i_0\}\) and \(w \ne i_0\) unmarked in \(\mathbf{i}'.\) 
\end{lemma}

\underline{Step \(6'.\)} Take into account the conditional expectation coming from \(\mathbf{i}'':\) the overall contribution of cycles of types \((I)\) and \((II)\) is at most
\begin{equation}\label{iandii}
    c(\kappa,c) \cdot 4m n^{-\delta}e^{16}\sum_{1 \leq l \leq p-1}{\binom{2p}{2l}(2m)^{2l}} \cdot 2^{2p-2l}((2l+2)!)^{4l}M^{2l},
\end{equation}
and cycles of type \((III)\) give
\begin{equation}\label{spm}
    2m\sum_{1 \leq l \leq p-1}{M^{2l}\sum_{1 \leq t \leq p-l+1, 0 \leq l_0 \leq \min{(\frac{t}{2},l)}}{\binom{l-l_0+t-1}{l-l_0}\binom{t}{2l_0}}b_{p-l,t}}:=2m \cdot (s(p,M)-M^{2p}),
\end{equation}
yielding (\ref{condtrace}) (\(l=p\) generates the extra term \(M^{2p}:\) such \(\mathbf{i}''\) are fully determined by their first edge because by the conditioning, the very sparse \(A_{B,\kappa}\) has no row with two nonzero entries), with this last bound tight, up to polynomial factors in \(m,p.\)
\par
The next driving force behind these conditional expectations is the size of \(b_{l,t},\) encapsulated by inequality \((24)\) in \cite{oldpaper}:  
\begin{equation}\label{bsizes}
    \frac{1}{4l} \cdot \binom{2l+1-t}{l} \leq b_{l,t} \leq  (l+1)^{120} \cdot \binom{2l+1-t}{l}
\end{equation}
for \(1 \leq t \leq l+1.\) Using these binomial proxies for \(b_{l,t}\) (any polynomial factor becomes negligible when \(p \to \infty\)), the limiting behavior of \(s(p,M)\) can be fully determined: Stirling's formula entails
\begin{equation}\label{stirling}
    C_1\sqrt{n} \cdot (\frac{n}{e})^n \leq n! \leq C_2\sqrt{n} \cdot (\frac{n}{e})^n
\end{equation}
for universal constants \(C_1,C_2>0\) and all \(n \in \mathbb{N},\) whereby
\[\lim_{p \to \infty}{(s(p,M))^{1/p}}=\sup{[M^{2x} \cdot \frac{(x-y+z)^{x-y+z}}{(x-y)^{x-y}(2y)^{2y}(z-2y)^{z-2y}} \cdot \frac{(2-2x-z)^{2-2x-z}}{(1-x)^{1-x}(1-x-z)^{1-x-z}}]}=f^2(M)\]
over \(0 \leq x \leq 1, 0 \leq z \leq 1-x, 0 \leq y \leq \min{(\frac{z}{2},x)}\) (take \(l=px,l_0=py,t=pz:\) (\ref{spm}) gives the size of \(s(p,M),\) up to polynomial factors in \(p,\) which are negligible when taking the \(p^{th}\) root). The supremum of the continuous function in the middle can be easily obtained since it is \(C^1\) in the interior of its domain (lemma \(6\) in \cite{oldpaper}; more suprema of this form are computed below: see Lemmas~\ref{lemma79} and \ref{lemma9}).
\par
The last missing pieces are the variances of such trace differences, (\ref{condvar}) and (\ref{condvar2}): once again even cycles play a central role by reasoning in the same vein as Sinai and Soshnikov~\cite{sinaisosh} do when analyzing the variance of moments of the trace of large powers of Wigner matrices (see subsection \(3.2\) in \cite{oldpaper}). Clearly,
\[Var_{**}[(tr((A_{B,\kappa}+A_s)^{2p})-tr(A^{2p}_s)]=\sum_{(\mathbf{i},\mathbf{j})}{(\mathbb{E}_{**}[a_{\mathbf{i}} \cdot a_{\mathbf{j}}]-\mathbb{E}_{**}[a_{\mathbf{i}}] \cdot \mathbb{E}_{**}[a_{\mathbf{j}}])},\]
where \(\mathbf{i},\mathbf{j}\) are cycles of length \(2p\) containing at least one edge belonging to \(A_{B,\kappa}.\) By independence, the contribution of \((\mathbf{i},\mathbf{j})\) is nonzero only if \(\mathbf{i}\) and \(\mathbf{j}\) share at least one undirected edge belonging to \(A_{s},\) and every undirected edge in their union \(\mathbf{i} \cup \mathbf{j}\) appears an even number of times. 
\par
A crucial step in \cite{sinaisosh} is mapping such pairs \((\mathbf{i},\mathbf{j})\) to even cycles \(\mathcal{P}\) of length \(2 \cdot 2p-2=4p-2:\) this is achieved by merging these two cycles along a common edge \(e,\) which then gets erased (\(\mathcal{P}\) traverses \(\mathbf{i}\) up to this shared edge \(e\), which is used as a bridge to switch to \(\mathbf{j},\) traverse all of it, and get back to the rest of \(\mathbf{i}\) upon returning to \(e\)). Evidently, \(\mathcal{P}\) is an even cycle of length \(2 \cdot 2p-2=4p-2,\) and deleting the shared edge \(e\) can be done at the cost of at most a factor of \(n^{2\delta},\) generating some decay when divided by \(n\) (by design, \(\delta<1/2\)), and a bound similar to (\ref{condtrace}) completes the proof. This gluing will also be employed for Theorems~\ref{th3} and \ref{th1}.

\subsection{Traces with Lacunary Cycles}\label{0.2}

As mentioned previously, the proofs of the results rely on the technique described in subsection \ref{0.1}, whose gist consists of finding \(S,Q\) with \(||S+Q-P||\) small, \(S\) sparse, and \(||Q|| \leq ||S+Q||.\) In our case, the same truncation is employed (i.e., based on the sizes of the entries of \(A\)), and the trace difference of interest becomes 
\[\mathbb{E}_{(*)}[tr((P_s+P_{B,\kappa})^{p})-tr(P_s^{p})].\]
\par
If \(||v||_{\infty}=o(1),\) then the entries of \(P_{B,\kappa}\) do not differ much from those of \(A_{B,\kappa}.\) In contrast, when \(v\) has some entries bounded away from zero, \(P_{B,\kappa}\) has some additional nonzero ones besides those from \(A_{B,\kappa}\) (in \((P_{B,\kappa})_{ij},\) where \(v_i \ne 0, v_j \ne 0,\) the random component is dominated by \(\theta v_i v_j\) with high probability). By freezing them in the conditioning, one is left with segments of a cycle whose entries are elements of \(P_s.\)
\par
Suppose next one wished to compute \(\mathbb{E}[tr(P_s^p)]:\) to do so, write each matrix entry as a sum of two elements, open the parentheses of these products, and obtain thus a change of summation, over cycles of length \(p\) in which each factor is either an entry of \(A_s\) or \(\theta v_iv_j\chi_{|a_{ij}| \leq n^{1/4-\delta_1}}\) for some \(i,j.\) Call the latter edges \textit{erased}: because such random variables are easy to control when the rest of the cycle is fixed, the summation becomes straightforward: up to the probabilities coming from the indicator functions, their overall contributions are powers of \(\theta\) times \(\prod{(v_{l_i}v_{r_i})},\) with the product taken over the (left and right) endpoints of the erased clusters. Returning to the entries belonging to \(A_s,\) symmetry yields their contribution vanishes unless each undirected edge has even multiplicity, which implies they can be arranged in an even cycle. This observation and fine bounds on \(\mathbb{E}[tr(A_s^{2p})]\) allow then to obtain the exact size of the trace (up to negligible polynomial factors when taking the \(p^{th}\) root). 
\par
In the case \(\alpha=4,\) a careful control of the difference above is needed (when \(\alpha \in (0,4),\) this difference can be avoided and replaced by \(tr((P_s+P_{B,\kappa})^{2p})\) as solely an upper bound is needed, the lower bound being already tight) because, as it can be seen in the unperturbed case, there are non-negligible components (i.e., the cycles of type \((III)\)). In this situation, the same occurs, with even more such contributors (the moments of the entries can be larger due to the perturbation), and their counting depends on the even powers of the entries of \(v.\) If \(||v||_{\infty}=o(1),\) then \(\sum_{1 \leq i \leq n}{v_i^{2p}} \leq ||v||_{\infty}^{2p-2},\) rendering the new cycles not difficult to track (if among the endpoints, one has multiplicity at least \(4,\) then it generates some decay), with the dominating ones producing approximately \((F(\theta))^{2p}.\) In contrast, the entries of \(v\) bounded away from zero create large entries in \(P_{B,\kappa},\) which in turn generate a class of cycles making an overall contribution of roughly \((f(\theta))^{2p}.\)

\section{Heavy-Tailed Distributions with \(\alpha\in(0,4)\)}\label{sect1}

Fix \(\delta \in (0,\frac{1}{2\alpha}),\kappa>0, x>0,\)
\[x \in \begin{cases}
    (\frac{1}{\alpha}-\frac{1}{2\alpha^2},\frac{1}{\alpha}), & 0<\alpha \leq 2 \\
    (\frac{\alpha-2}{\alpha(\alpha-1)},\frac{1}{4}) \cap (\frac{3\alpha-8}{2\alpha(\alpha-2)},\frac{1}{4}), & 2<\alpha<4
    \end{cases}\]
(the intersection is non-empty as \(\alpha^2-5\alpha+8>0,(\alpha-4)^2>0\)), and consider the decomposition
\[P=(p_{ij}\chi_{|a_{ij}| \leq n^x})+(p_{ij}\chi_{n^x<|a_{ij}| \leq n^{\frac{3}{2\alpha}+\delta}})+(p_{ij}\chi_{|a_{ij}|>n^{\frac{3}{2\alpha}+\delta}}) :=P_s+P_m+P_b,\]
\[P_b=(p_{ij}\chi_{n^{\frac{3}{2\alpha}+\delta}<|a_{ij}|<\kappa b_n})+(p_{ij}\chi_{|a_{ij}| \geq \kappa b_n}):=P_{b,\kappa}+P_{B,\kappa}.\]
Showing for any \(\epsilon>0\) 
\begin{equation}\label{negligiblematrices}
    ||P_{b,\kappa}|| \leq 2\kappa, \hspace{0.3cm} ||P_m||=o_p(1),
\end{equation}
\begin{equation}\label{lowerbound2}
    \lim_{n \to \infty}{\mathbb{P}(\lambda_1(P) \geq \max{(max(A),\theta)}-\epsilon)}=1,
\end{equation}
\begin{equation}\label{upperbound}
    \mathbb{P}(||P_s+P_{B,\kappa}||>\max{(M,\theta)}+\epsilon) \leq \mathbb{P}(max(A)>M)+o(1)
\end{equation}
suffices since \(\kappa>0\) can be chosen arbitrarily small (the inequality in (\ref{negligiblematrices}) holding with high probability). The coming three subsections justify these claims in the stated order.

\subsection{Negligible Components: \(P_m\) and \(P_{b,\kappa}\)}\label{7.1}

Let
\[P_m=b^{-1}_n(a_{ij}\chi_{n^x<|a_{ij}| \leq n^{\frac{3}{2\alpha}+\delta}})+(\theta v_iv_j\chi_{n^x<|a_{ij}| \leq n^{\frac{3}{2\alpha}+\delta}}):=A_m+P_1,.\] 
\[P_{b,\kappa}=b^{-1}_n(a_{ij}\chi_{n^{\frac{3}{2\alpha}+\delta}<|a_{ij}| \leq \kappa b_n})+(\theta v_iv_j\chi_{n^{\frac{3}{2\alpha}+\delta}<|a_{ij}| \leq \kappa b_n}):=A_{b,\kappa}+P_2.\]
Proving \(||A_m||=o_p(1),||P_1||=o_p(1),||P_2||=o_p(1),\) and \(||A_{b,\kappa}|| \leq \kappa\) with high probability entails (\ref{negligiblematrices}). The last claim is immediate because \(A_{b,\kappa}\) has at most one nonzero entry per row with high probability, whereby \(||A_{b,\kappa}|| \leq \max_{1 \leq i,j \leq n}{|(A_{b,\kappa})_{ij}|} \leq \kappa:\) the complement of this event has probability at most
\[n \cdot n^2 \cdot (\mathbb{P}(|a_{11}|>n^{\frac{3}{2\alpha}+\delta}))^2 \leq n^3 \cdot (n^{-(\frac{3}{2\alpha}+\delta) \cdot (\alpha-\epsilon)})^2=n^{3-(3+2\alpha \delta)(1-\frac{\epsilon}{\alpha})}=o(1)\]
for an \(\epsilon>0\) sufficiently small.
\par
Consider \(A_m:\) in light of \(||A_m|| \leq b_n^{-1}\max_{1 \leq i \leq n}{\sum_{1 \leq j \leq n}{|a_{ij}|\chi_{n^x<|a_{ij}| \leq n^{\frac{3}{2\alpha}+\delta}}}},\) justifying 
\begin{equation}\label{easypart}
     b_n^{-1}\max_{1 \leq i \leq n}{\sum_{1 \leq j \leq n}{|a_{ij}|\chi_{n^x<|a_{ij}| \leq n^{\frac{3}{2\alpha}+\delta}}}}=o_p(1)
\end{equation}
gives the desired result. If \(\alpha \leq 2,\) part \((c)\) of proposition \(5.6\) in Benaych-Georges and Péché~\cite{benaychpeche} yields for fixed \(i,\)
\[\sum_{1 \leq j \leq n}{|a_{ij}|\chi_{n^x<|a_{ij}| \leq n^{\frac{3}{2\alpha}+\delta}}} \leq n^{\frac{1}{\alpha}+\alpha(\frac{1}{\alpha}-x)+\frac{1}{2\alpha}+\gamma+\gamma'}\]
for any \(\gamma'>0,\) with very high probability (i.e., the probability of its complement decays at least exponentially in \(n\)) using \(\mu=1,x=\frac{1}{\alpha}-\eta,\eta'=\frac{1}{2\alpha}+\gamma, \epsilon=\alpha(\frac{1}{\alpha}-x)+\frac{1}{2\alpha}+\gamma+\gamma';\)
finally,
\[\frac{1}{\alpha}+\alpha(\frac{1}{\alpha}-x)+\frac{1}{2\alpha}+\gamma+\gamma'<\frac{2}{\alpha}-2\gamma'',\]
for some \(\gamma''>0\) because \(\alpha(\frac{1}{\alpha}-x)+\gamma+\gamma'<\frac{1}{2\alpha};\) a union bound and \(b_n \geq n^{2\alpha-\gamma''}\) render (\ref{easypart}).
If \(\alpha \in (2,4),\) then \[b_n^{-1}n\mathbb{E}[|a_{11}|\chi_{n^x<|a_{11}| \leq n^{\frac{3}{2\alpha}+\delta}}] \leq n^{1-\frac{2}{\alpha}+\epsilon+(1-\alpha+\epsilon)x}=o(1)\] 
using \(x>\frac{\alpha-2}{\alpha(\alpha-1)},\) and 
\[b_n^{-4}n^2\mathbb{E}[a^4_{11}\chi_{n^x<|a_{11}| \leq n^{\frac{3}{2\alpha}+\delta}}] \leq n^{2-\frac{8}{\alpha}+\epsilon+(-\alpha+4+\epsilon)(\frac{3}{2\alpha}+\delta)}=o(1),\]
\[b_n^{-4}n^3(\mathbb{E}[a^2_{11}\chi_{n^x<|a_{11}| \leq n^{\frac{3}{2\alpha}+\delta}}])^2 \leq n^{3-\frac{8}{\alpha}+\epsilon+(-\alpha+2+\epsilon)2x}=o(1),\]
as \(2-\frac{8}{\alpha}+(-\alpha+4) \cdot (\frac{3}{2\alpha}+\delta)=(\delta-\frac{1}{2\alpha})(4-\alpha)<0, x>\frac{3\alpha-8}{2\alpha(\alpha-2)};\) (\ref{easypart}) ensues from a union bound and Markov's inequality for the fourth moment (after centering the left-hand side term).
\par
Continue with \(P_1\) and \(P_2:\) for \(p \in \mathbb{N},\) 
\[\mathbb{E}[tr(P_1^{p})] \leq \theta^{p}\mathbb{P}(n^x<|a_{11}| \leq n^{\frac{3}{2\alpha}+\delta}),\]
\[\mathbb{E}[tr(P_2^{p})] \leq \theta^{p}\mathbb{P}(n^{\frac{3}{2\alpha}+\delta}<|a_{11}| \leq \kappa b_n)\]
since if \(P_1=(m_{ij})_{1 \leq i,j \leq n},\) then
\[\mathbb{E}[tr(P_1^{p})]=\sum_{(i_0,i_1, \hspace{0.05cm} ... \hspace{0.05cm} ,i_{p-1}) \in \{1,2, \hspace{0.05cm} ... \hspace{0.05cm} ,n\}^p}{\mathbb{E}[m_{i_0i_1}m_{i_1i_2}...m_{i_{p-1}i_0}]} \leq \theta^{p}\mathbb{P}(n^x<|a_{11}| \leq n^{\frac{3}{2\alpha}+\delta})\]
as \((\chi_{n^x<|a_{ij}| \leq n^{\frac{3}{2\alpha}+\delta}})^l=\chi_{n^x<|a_{ij}| \leq n^{\frac{3}{2\alpha}+\delta}}\) for \(l \in \mathbb{N},\) whereby for \(i_0,i_1, \hspace{0.05cm} ... \hspace{0.05cm} ,i_{p-1} \in \{1,2, \hspace{0.05cm} ... \hspace{0.05cm} ,n\},\)
\[|\mathbb{E}[m_{i_0i_1}m_{i_1i_2}...m_{i_{p-1}i_0}]| \leq \mathbb{P}(n^x<|a_{11}| \leq n^{\frac{3}{2\alpha}+\delta}) \prod_{0 \leq j \leq p-1}{\theta |v_{i_j}| \cdot |v_{i_{j-1}|}}=\theta^p\mathbb{P}(n^x<|a_{11}| \leq n^{\frac{3}{2\alpha}+\delta}) \prod_{0 \leq j \leq p-1}{v^2_{i_j}},\]
with an identical rationale applying to \(P_2\) as well. Markov's inequality yields for \(t>0, p=\lfloor \log{\log{n}} \rfloor,\)
\[\mathbb{P}(||P_1|| \geq t) \leq \mathbb{E}[tr((P_1)^{2p})]/t^{2p} \leq n^{-3/2}(\theta/t)^{2p}=o(1),\]
\[\mathbb{P}(||P_2|| \geq t) \leq \mathbb{E}[tr((P_2)^{2p})]/t^{2p} \leq n^{-x\alpha/2}(\theta/t)^{2p}=o(1),\]
completing the justification of (\ref{negligiblematrices}).

\subsection{Lower Bound}\label{7.2}

Split (\ref{lowerbound2}) into
\begin{equation}\label{maxeq}
    \lim_{n \to \infty}{\mathbb{P}(\lambda_1(P) \geq max(A)-\epsilon)}=1,
\end{equation}
\begin{equation}\label{thetaeq}
    \lim_{n \to \infty}{\mathbb{P}(\lambda_1(P) \geq \theta-\epsilon)}=1.
\end{equation}
\par
Since \(\theta v v^T\) is positive semidefinite, \(\lambda_1(P) \geq \lambda_1(b_n^{-1}A),\) and \(\lambda_1(b_n^{-1}A) \geq max(A)-\epsilon\) because for \(\epsilon>0\) small enough,
\[\max_{1 \leq i \leq n}{|a_{ii}|} \leq b_n^{1-2\epsilon} \leq b_n^{1-\epsilon} \leq \max_{1 \leq i \leq j \leq n}{|a_{ij}|},\] 
with high probability (\(b_{\sqrt{2n}}^{-1}\max_{1 \leq i \leq n}{|a_{ii}|} \Rightarrow E_{\alpha}\)), in which case \(|a_{i_0j_0}|=\max_{1 \leq i,j \leq n}{|a_{ij}|}:=b_n max(A)\) has \(i_0 \ne j_0\) and \(u \in \mathbb{S}^{n-1}\) with \(u_{i_0}=|u_{j_0}|=\frac{1}{\sqrt{2}},u_{j_0}a_{i_0j_0} \geq 0\) entails
\[b_n^{-1}u^TAu=b_n^{-1}|a_{i_0j_0}|+\frac{1}{2}b_n^{-1}(a_{i_0i_0}+a_{j_0j_0}) \geq max(A)-b_n^{-1}\max_{1 \leq i \leq n}{|a_{ii}|} \geq max(A)-b_n^{-\epsilon}.\]
\par
With (\ref{maxeq}) complete, proceed with (\ref{thetaeq}): take \(\hat{A}_{b,L}=b^{-1}_n(a_{ij}\chi_{n^{\frac{3}{2\alpha}+\delta} < |a_{11}| \leq Lb_n}),\hat{A}_{B,L}=b^{-1}_n(a_{ij}\chi_{|a_{11}|> Lb_n})\)
\[v^TPv=\theta+b_n^{-1}\sum_{1 \leq i,j \leq n}{v_iv_ja_{ij}}=\theta+v^TA_sv+v^TA_mv+v^T\hat{A}_{b,L}v+v^T\hat{A}_{B,L}v.\]
Chebyshev's inequality gives \(v^TA_sv, v^TA_{b,L}v=o_p(1):\) their expectations are \(0\) by symmetry, and \[\mathbb{E}[(v^TA_sv)^2]=b_n^{-2}\mathbb{E}[a^2_{11}\chi_{|a_{11}| \leq n^x}] \leq n^{2x}b^{-2}_{n},\]
\[\mathbb{E}[(v^TA_{b,L}v)^2]=b_n^{-2}\mathbb{E}[a^2_{11}\chi_{n^{\frac{3}{2\alpha}+\delta}<|a_{11}| \leq Lb_n}] \leq L^2\mathbb{P}(|a_{11}| > n^{\frac{3}{2\alpha}+\delta}).\] 
Additionally, \(||A_m||=o_p(1)\) from (\ref{negligiblematrices}), whereby
\[\mathbb{P}(\lambda_1(P) < \theta-\epsilon) \leq \mathbb{P}(max(A) \geq L)+o(1)+(n^{2x}b^{-2}_{n}+L^2n^{-3/2})(\epsilon/3)^{-2}\]
for \(L>0,\) rendering (\ref{thetaeq}) inasmuch as \(max(A) \Rightarrow E_{\alpha}.\)

\subsection{Upper Bound}\label{7.3}

Traces of large even powers of \(P_s+P_{B,\kappa}\) are employed to justify (\ref{upperbound}): since a conditioning as in subsection~\ref{0.1} would require \(b_n\) to grow faster than regular variation imposes (\((\mathbb{P}(|a_{11}| \leq \kappa b_n))^{-n^2}\) must diverge slower than a power of \(n,\) which translates to \(\liminf_{n \to \infty}{\frac{b_n}{\frac{n^{2/\alpha}}{(\log{n})^{1/\alpha}}}}<\infty\)), work instead with indicator functions.
\par
Fix \(\epsilon_0>0, \epsilon_1=\epsilon_1(\epsilon_0) \in (0,\min{(\alpha,\frac{2}{\alpha})}), (2/\alpha-\epsilon_1)(\alpha-\epsilon_1)=2-\epsilon_0\) (for some \(c(\alpha)>0\) and all \(\epsilon_0 \in (0,c(\alpha)),\) there is such \(\epsilon_1:\) in the end, \(\epsilon_0 \to 0\)); let \(V=\{1 \leq i \leq n: |v_i| \geq n^{-1/4}\},m=m_n=n^{2\epsilon_0},\) and \(\mathcal{S}=\mathcal{S}_{n,m}\) the set of \(S \subset \{(i,j): 1 \leq i< j \leq n\}\) with:
\par
\((a) \hspace{0.05cm} |S| \leq m;\) 
\par
\((b)\) any \((i,j) \in S\) has \(\{i,j\} \cap V=\emptyset;\)
\par
\((c)\) all pairwise distinct elements \((i_1,j_1),(i_2,j_2)\) of \(S\) satisfy \(\{i_1,j_1\} \cap \{i_2,j_2\}=\emptyset. \newline\) 
Consider the events
\[E(S,\kappa,M)=\{\max_{i \leq j, (i,j) \not \in S}{|a_{ij}|} \leq \kappa b_n<\min_{(i,j) \in S}{|a_{ij}|} \leq \max_{(i,j) \in S}{|a_{ij}|} \leq Mb_n\},\]
(i.e., the set of positions of the nonzero entries of \(A_{B,\kappa}\) is \(S \cup \{(j,i): (i,j) \in S\}\)). Clearly, \((E(S,\kappa,M))_{S \in \mathcal{S}}\) are pairwise disjoint, and the complement of \(\cup_{S \in \mathcal{S}}{E(S,\kappa,M)}\) is included in the union of the following events: there are at least \(2m+1\) entries of \(A\) larger (in absolute value) than \(\kappa b_n\) or two of them on the same row, \(max(A)>M,\) \(b_n^{-1}\max_{1 \leq i \leq n}{|a_{ii}|} \geq \kappa,\) and \(b_n^{-1}\max_{1 \leq i<j \leq n, \{i,j\} \cap V \ne \emptyset}{|a_{ij}|} \geq \kappa,\) whose probabilities add up to at most
\[\binom{n^2}{m}(\mathbb{P}(|a_{11}| \geq \kappa b_n))^m+n  \cdot n^2(\mathbb{P}(|a_{11}| \geq \kappa b_n))^2+\mathbb{P}(max(A)>M)+o(1)=\mathbb{P}(max(A)>M)+o(1)\]
because
\[\binom{n^2}{m}(\mathbb{P}(|a_{11}| \geq \kappa b_n))^m \leq \frac{n^{2m}}{m!} (\kappa n^{2/\alpha-\epsilon_1})^{-(\alpha-\epsilon_1)m} \leq \frac{n^{2m}}{(m/e)^{m}} (\kappa n^{2/\alpha-\epsilon_1})^{-(\alpha-\epsilon_1)m}=(\kappa^{-(\alpha-\epsilon_1)}en^{-\epsilon_0})^m,\]
and for \(\epsilon_2>0\) small enough, with high probability \(b_n^{1-\epsilon_2} \leq max(A),\) whereas 
\[\max_{1 \leq i \leq n}{|a_{ii}|}\leq n^{1/\alpha+\epsilon_2} \leq b_n^{1-2\epsilon_2}, \max_{1 \leq i<j \leq n, \{i,j\} \cap V \ne \emptyset}{|a_{ij}|} \leq (2n|V|)^{1/\alpha+\epsilon_2} \leq (2n^{3/2})^{1/\alpha+\epsilon_2} \leq b_n^{1-2\epsilon_2}.\]
\par
A sufficient condition for (\ref{upperbound}) is for \(p \leq c(M,\alpha)\log{n},\)
\begin{equation}\label{condtraceM1}
    \mathbb{E}[tr((P_s+P_{B,\kappa})^{2p})\chi_{E(S,\kappa,M)}] \leq \theta^{2p}+n^{-c_{1}(\alpha)}:
\end{equation}
uniformly in \(S \in \mathcal{S}.\) These inequalities and Markov's inequality yield for \(p=\lfloor c(\epsilon,\theta) \log{n} \rfloor,\) 
\[\mathbb{P}(||P_s+P_{B,\kappa}||>\max{(M,\theta)}+\epsilon) \leq \mathbb{P}((\cup_{S \in \mathcal{S}}{E(S,\kappa,M)})^c)+ (\frac{\theta}{\theta+\epsilon})^{2p}+n^{-c_{1}(\alpha)}(\theta+\epsilon)^{-p}=\mathbb{P}(max(A)>M)+o(1).\]
\par
Proceed now with (\ref{condtraceM1}): since
\[\chi_{E(S,\kappa,M)}=\prod_{1 \leq i \leq j \leq n, (i,j) \not \in S}{\chi_{|a_{ij}| \leq \kappa b_n}} \prod_{1 \leq i \leq j \leq n, (i,j) \in S}{\chi_{|a_{ij}| \in (\kappa b_n,Mb_n]}},\]
the cycles underlying the expectation have few edges with large moments (i.e., for \((i,j) \in S,\) the indicator is \(\chi_{|a_{ij}| \in (\kappa b_n,Mb_n]},\) while for the rest it is \(\chi_{|a_{ij}| \leq n^x}\)). When \(\theta=0,\)
\[tr((P_s+P_{B,\kappa})^{p})=\sum_{\mathbf{i}=(i_0,i_1, \hspace{0.05cm} ... \hspace{0.05cm},i_{p-1},i_0) \in A_{s} \cup A_{B,\kappa}}{a_{i_0i_1}a_{i_1i_2}...a_{i_{p-1}i_0}},\]
where the notation in the summation is a shorthand for all edges in the cycle belonging to \(A_s\) or \(A_{B,\kappa}.\) To deal with the more general situation \(\theta>0,\) open the parentheses in each product corresponding to such \(\mathbf{i}:\) the new summands still correspond to cycles, but some edge contributions are replaced as in \(a_{ij} \to \theta v_iv_j\chi_{|a_{ij}| \in (0,n^x]}\) or \(a_{ij} \to \theta v_iv_j\chi_{|a_{ij}| \in \cup (\kappa b_n,Mb_n]}.\) Split these products in three classes by seeing them as polynomials in \(\theta\) with:
\begin{center}
    (C1) no factor of \(\theta,\hspace{0.3cm}\) (C2) at least one factor of \(\theta\) and an \(a_{ij}, \hspace{0.3cm}\) (C3) only factors of \(\theta.\) 
\end{center}
\par
Consider C3, and argue as for \(P_1\) and \(P_2\) in subsection~\ref{7.1}. An undirected edge \(e=ij\) of multiplicity \(l\) in \(\mathbf{i}\) generates a factor that is at most
\[|\theta v_iv_j|^l\mathbb{E}[(\chi_{|a_{ij}| \in (0,n^x] \cup (\kappa b_n,Mb_n]})^l]=|\theta v_iv_j|^l\mathbb{P}(|a_{11}| \in (0,n^x] \cup (\kappa b_n,Mb_n]]) \leq |\theta v_iv_j|^l,\]
yielding the overall contribution is upper bounded by \(\theta^{2p}.\)
\par
Continue with C1, for which it is shown for \(p \leq n^{c_1(\alpha,M)},\)
\begin{equation}\label{nothetatrace2}
    \mathbb{E}[tr((A_s+A_{B,\kappa})^{2p})\chi_{E(S,\kappa,M)}] \leq n^{-c_2(\alpha)}:
\end{equation}
this plays a role in the analysis of C2 terms as well. Proceed as in the case \(\alpha=4:\) for \(\epsilon>0, n \geq n(\epsilon),\) 
\begin{equation}\label{mb1}
    \mathbb{E}[a^{2q}\chi_{|a_{11}| \leq n^x}] \leq \max{(n^{(2q-\alpha+\epsilon)x},C)}
\end{equation}
since
\[\mathbb{E}[a^{2q}\chi_{|a_{11}| \leq n^x}] \leq (C(\epsilon))^{2q}+\int_{C(\epsilon)}^{n^x}{2qy^{2q-1} \cdot y^{-\alpha+\epsilon/2}dy} \leq (C(\epsilon))^{2q}+n^{x(2q-\alpha+\epsilon/2)} \cdot \frac{2q}{2q-\alpha+\epsilon/2} \leq n^{x(2q-\alpha+\epsilon)}\]
for \(q \geq 2>\frac{\alpha}{2};\) this inequality remains true for \(q=1, \alpha \leq 2,\) and when \(q=1, \alpha>2,\) the expectation is at most \(C=\mathbb{E}[a^2_{11}].\)
\par
Take first the cycles containing only elements belonging to \(A_s,\) and reason as in subsection \(2.1\) of \cite{oldpaper}. The analogue of the inequality proved in it for \(B=A_s\) is:

\begin{lemma}\label{lemma0}
There exist \(c_3(\alpha), c_4(\alpha),c_5(\alpha)>0\) such that for \(p \in \mathbb{N}, p \leq n^{c_3(\alpha)},\) 
\begin{equation}\label{boundBtrace}
    \mathbb{E}[tr(B^{2p})] \leq n^{2-c_4(\alpha)-pc_5(\alpha)}.
\end{equation}
\end{lemma}

\begin{proof}
Use the counting technique described in subsection~\ref{0.1} underlying (\ref{trace}). Steps \(1-4\) continue to hold as they are concerned with combinatorial properties of even cycles, whereas in step \(5\) the moment bounds change. For an undirected edge \(e,\) denote by \(2k(e)\) its multiplicity in \(\mathbf{i}\) (\(k(e) \in \mathbb{Z}\) because \(\mathbf{i}\) is even), and let \(E=E(\mathbf{i}):=|\{e, k(e) \geq 2\}|.\) If \(\alpha \leq 2,\) then (\ref{mb1}) gives
\[\mathbb{E}[b_{\mathbf{i}}] \leq \prod_{e, k(e)=1}{n^{(2-\alpha+\epsilon)x}} \prod_{e, k(e) \geq 2}{n^{(2k(e)-\alpha+\epsilon)x}}=n^{x(2-\alpha+\epsilon)(p-\sum{k(e)})}n^{2x\sum{k(e)}-x(\alpha-\epsilon) E}.\]
The key to control this expectation is inequality \((14)\) in \cite{oldpaper}:
\begin{equation}\label{eq14}
    \sum_{e,k(e) \geq 2}{k(e)} \leq E+1+\sum_{k \geq 2}{kn_k},
\end{equation}
giving
\[\mathbb{E}[b_{\mathbf{i}}] \leq n^{x(2-\alpha+\epsilon)p+x(\alpha-\epsilon)\sum_{k \geq 2}{kn_k}+x(\alpha-\epsilon)} \leq n^{x(2-\alpha+\epsilon)p+x(\alpha-\epsilon)p+x(\alpha-\epsilon)}.\]
Reasoning in the same fashion as in section \(2.1\) of \cite{oldpaper} yields the bound
\[b_n^{-2p} \cdot n^{x(2-\alpha+\epsilon)p+x(\alpha-\epsilon)p+x(\alpha-\epsilon)} \cdot n^{p+1}2^{2p}e^{p},\]
providing the desired results because \(2x+1-\frac{4}{\alpha}<1-\frac{2}{\alpha} \leq 0.\)
\par
If \(\alpha \in (2,4),\) then
\[\mathbb{E}[b_{\mathbf{i}}] \leq \prod_{e, k(e)=1}{C} \prod_{e, k(e) \geq 2}{n^{(2k(e)-\alpha+\epsilon)x}}=C^{p-\sum{k(e)}}n^{2x\sum{k(e)}-x(\alpha-\epsilon) E} \leq (C+1)^{p}n^{2x+2x\sum_{k \geq 2}{kn_k}},\]
and since \(x<\frac{1}{4},\) the overall bound is for \(p=o(n^{1/2-2x}),\)
\[b_n^{-2p} \cdot (C+1)^{p} n^{2x}\cdot 2^{2p}n^{p+1}e^{8},\]
in which case the conclusion follows as \(\alpha<4.\)
\end{proof}

Return to the analysis of C1 terms. As in subsection~\ref{0.1}, split an even cycle \(\mathbf{i}\) into \(\mathbf{i}',\mathbf{i}'',\) fix \(\mathbf{i}'',\) and denote by \(2l\) its number of edges (if \(l \not \in \mathbb{Z},\) its contribution vanishes by symmetry). Lemma~\ref{lemma0} and (\ref{iandii}) give the bound
\[\sum_{1 \leq l \leq p-1, 1 \leq t \leq l}{\binom{2p}{2l}(2m)^{t} (2t)^{2l} n^{-t} M^{2l}n^{(-2+\epsilon)t} \cdot n^{2-c_4(\alpha)-(p-l)c_5(\alpha)}(p-l)!}:\]
let \(t\) be the number of pairwise distinct undirected edges appearing in \(\mathbf{i}''\) (\(t \leq l\) since \(\mathbf{i}\) is even); for a given \(t,\) there are at most \(\binom{2p}{2l}(2m)^{t} (2t)^{2l}\) such cycles, for which \(\mathbb{E}[a_{\mathbf{i}''}] \leq M^{2l}n^{(-2+\epsilon)t}\) because
\[b_n^{-2q}\mathbb{E}[a^{2q}_{11}\chi_{|a_{11}| \in (\kappa b_n,Mb_n]}] \leq M^{2q}\mathbb{P}(|a_{11}|>\kappa b_n) \leq M^{2q}n^{-2+\epsilon};\]
the factor \(((2l+2)!)^{4l}\) from step \(3',\) counting the possibilities in which certain unmarked edges of \(\mathbf{i}'\) can be arranged, might be replaced by \((p-l)!\) (\(\mathbf{i}'\) has length \(2(p-l)\)); finally, \(V(\mathbf{i}')\) contains at least \(t\) vertices from \(\mathbf{i}''\) and hence a saving of \(n^{-t}\) in step \(3'.\) Since
\[\frac{(2m)^{t+1}(2t+2)^{2l}n^{-(t+1)}n^{(-2+\epsilon)(t+1)}}{(2m)^{t}(2t)^{2l}n^{-t}n^{(-2+\epsilon)t}} \leq 2m \cdot 2^{2l} \cdot n^{-1} \cdot n^{-2+\epsilon}<1,\]
for \(m=n^{2\epsilon_0}, p \leq c(M,\alpha)\log{n},\) C1 contributes at most
\[\sum_{1 \leq l \leq p-1}{\binom{2p}{2l}l \cdot 2m  \cdot 2^{2l} \cdot n^{-1} M^{2l}n^{-2+\epsilon} \cdot n^{2-c_4(\alpha)-(p-l)c_5(\alpha)}(p-l)!} \leq\]
\begin{equation}\label{b1}
    \leq p\sum_{1 \leq l \leq p-1}{2m \cdot n^{-1} (2M) ^{2l}n^{-2+\epsilon} \cdot n^{2-c_4(\alpha)-(p-l)c_6(\alpha)}} \leq p^2((2M)^{2p-2}+1)n^{-1/2} \leq n^{-c_7(\alpha)}, 
\end{equation}
using \(\binom{2p}{2l}(p-l)! \leq (2p)^{2p-2l}(p-l)^{p-l} \leq (4p^3)^{p-l}.\)
\par
Continue with C2. For \(\mathbf{i}=(i_0,i_1, \hspace{0.05cm} ... \hspace{0.05cm},i_{2p-1},i_0),\) suppose there are \(s \geq 1\) clusters of erased edges with \((l_i,r_i,L_i)_{1 \leq i \leq s}\) their left, right endpoints, and their lengths, respectively (i.e., \(l_i \leq r_i-1, r_{i} \leq l_{i+1}-1, l_{s+1}:=2p-l_1\)). Symmetry yields \(\mathbb{E}[a_{\mathbf{i}}]=0\) if among the rest of the edges there is one with odd multiplicity, which makes such cycles negligible.
If \((l_i,r_i,L_i)_{1 \leq i \leq s}\) as well as the rest of \(\mathbf{i}\) are fixed, then the contribution of the \(j^{th}\) cluster is at most \(\theta^{L_j} |v_{l_j}v_{r_j}|\) because
\[\sum_{(i_1,i_2,\hspace{0.05cm} ... \hspace{0.05cm},i_{L_j-1})}{|(\theta v_{l_j}v_{i_1})(\theta v_{i_1}v_{i_2})...(\theta v_{L_j-1}v_{r_j})|}=\theta^{L_j}|v_{l_j}v_{r_j}|\sum_{(i_1,i_2, \hspace{0.05cm} ... \hspace{0.05cm},i_{L_j-1})}{v_{i_1}^2v_{i_2}^2...v^2_{i_{L_j-1}}}=\theta^{L_j}|v_{l_j}v_{r_j}|.\]
Hence they generate at most \(\max{(\theta^{2p},1)}\prod_{1 \leq j \leq s}{(|v_{l_j}v_{r_j}|)}.\) 
\par
Now consider the remaining \(s\) segments: recall each undirected edge appears an even number of times in their union. In particular, in the multiset of endpoints, each vertex has even multiplicity (the number of edges adjacent to it is odd for each segment). This renders these paths can be arranged in a union of disjoint even cycles (no two share an edge), by gluing them and using induction on \(s:\) if \(s=1,\) this is clear because \(l_1=r_1;\) suppose \(s \geq 2;\) if some \(l_i \ne r_i,\) then \(l_i=l_j\) or \(l_i=r_j\) for a \(j \ne i,\) and so the \(i^{th}\) and \(j^{th}\) segments can be merged; else, \(l_i=r_i,1 \leq i \leq s;\) if the first segment shares no edge with any other, then it is an even cycle disjoint from the rest; otherwise, it can be glued with another along a shared edge; in all three cases, the induction hypothesis can be employed.
\par
The previous paragraphs entail the C2 contribution can be upper bounded by products of even cycles with non-negative weights: look next at the preimages of the two mappings above (the first takes cycles to clusters and segments, and the second segments to sets of disjoint even cycles). For \(s\) fixed, these have sizes at most 
\begin{equation}\label{cp}
    2p \sum_{2p-s \geq l \geq s}{\binom{2p-l}{s-1} \cdot \binom{l-1}{s-1} \cdot s!} \leq 2p \sum_{2p-s \geq l \geq s}{\binom{2p-l}{s-1} \cdot l^s} \leq (2p)^2(2p-s)^{s-1}(2p-s)^{s} \leq (2p)^{2s+1},
\end{equation}
where \(l\) is the number of erased edges: clearly \(s \leq l \leq 2p-s\) since there are \(s\) erased clusters, each with at least one edge, and \(s\) remaining, with the same property; take a cycle of length \(2p-l\) (it must be even, but this can be left out); there are \(\binom{2p-l}{s-1}\) ways of splitting the cycle into \(s\) pieces, \(s!\) of ordering them, \(\binom{(l-s)+s-1}{s-1}=\binom{l-1}{s-1}\) possibilities for the distribution of erased edges (there are \(l\) of them in \(s\) clusters, each of size at least \(1,\) and 
\begin{equation}\label{atl}
    \tilde{a}_{t,L}=|\{(x_1,\hspace{0.05cm} ... \hspace{0.05cm}, x_t): x_1+...+x_{t}=L, x_i \in \mathbb{Z}, x_i \geq 0\}|=\binom{t+L-1}{L}=\binom{t+L-1}{t-1}
\end{equation}
by induction on \(L+t,\) employing \(\tilde{a}_{t,L}=\tilde{a}_{t-1,L}+\tilde{a}_{t,L-1}, L \geq 2\)), and \(2p\) ways of choosing the first vertex of the cycle. Let 
\[\tilde{S}:=\{1 \leq i \leq n, \hspace{0.1cm} \exists (i,j) \in S\} \cup \{1 \leq j \leq n, \hspace{0.1cm} \exists (i,j) \in S\},\]
and \(t_1,t_2\) be the numbers of endpoints (counting multiplicities) in \(\tilde{S}\) and \(\tilde{S}^c,\) respectively: the \(t_1\) vertices generate a factor of \(n^{-t_1/8}\) because \(|v_i| \leq n^{-1/4}, m=n^{2\epsilon_0},\) and so do the remaining \(t_2\) inasmuch as any even cycle with at least \(t_2\) vertices in \(\tilde{S}^c\) contains at least \(t_2/2\) edges belonging to \(A_s;\) thus, in (\ref{b1}) the sum is over \(1 \leq l \leq p-t_2/2,\) giving at most 
\(p^2((2M)^{2p-t_2}+1)n^{-tc_8(\alpha)-c_7(\alpha)}.\) 
\par
Therefore, putting all the endpoints together yields a factor of \(n^{-c_9(\alpha)s},\) whereby the overall bound becomes in light of (\ref{b1}),
\[\sum_{1 \leq s \leq p}{(2p)^{3s} \cdot p^2((2M)^{2p}+1)n^{-sc_8(\alpha)-c_7(\alpha)}} \leq p^2n^{-c_7(\alpha)}\sum_{1 \leq s \leq p}{n^{-c_9(\alpha)s} \cdot ((2M)^{2p}+1)} \leq\]
\[\leq 2p^3n^{-c_7(\alpha)}((2M)^{2p}+1) \leq n^{-c_{10}(\alpha)}\]
for \(p \leq c(M,\alpha)\log{n}.\)
\par

\section{Delocalized Eigenvector with \(\alpha=4\)}\label{sect3}

Let \(\delta_1, \delta_2, \delta_3 \in (0,\frac{1}{64}),\kappa>0\) be fixed, \((m_n) \subset \mathbb{N}, \lim_{n \to \infty}{m_n}=\infty, m^2_n \leq \min{(\log{\log{n}},||v||^{-1}_{\infty})},\) and 
\[\frac{1}{\sqrt{n}} A=A_s+A_m+A_{b,\kappa}+A_{B,\kappa},\]
\[A_s=\frac{1}{\sqrt{n}}(a_{ij}\chi_{|a_{ij}| \leq n^{1/4-\delta_1}}), \hspace{0.1cm} A_m=\frac{1}{\sqrt{n}}(a_{ij}\chi_{n^{1/4-\delta_1}<|a_{ij}| \leq n^{3/8+\delta_2}}),\] \[A_{b,\kappa}=\frac{1}{\sqrt{n}}(a_{ij}\chi_{n^{3/8+\delta_2}<|a_{ij}| \leq \kappa \sqrt{n} }), \hspace{0.1cm} A_{B,\kappa}=\frac{1}{\sqrt{n}}(a_{ij}\chi_{\kappa \sqrt{n} \leq |a_{ij}| }).\]
With high probability \(||A_s|| \leq 2+\epsilon, ||A_m||=o(1),||A_{b,\kappa}|| \leq \kappa, A_{B,\kappa}\) has at most \(m=m_n\) nonzero entries, any two on different rows (using for \(A_s\) (\ref{traceineq}) with \(p=\lfloor (\log{n})^2 \rfloor,\) reasoning as in \ref{7.1} for \(A_{b,\kappa}, A_{B,\kappa},\) and using a version of (\ref{traceineq}) for \(A_m:\) see subsection \(1.2\) of \cite{oldpaper}). 
\par
Two conditionings are of interest: define \(\mathcal{S}=\mathcal{S}_{n,m}, E(S,\kappa,M)\) as in subsection~\ref{7.3} (drop \((b)\)), and
\[E^{B}(S,\kappa,M)=E(S,\kappa,M) \cap \{a_{ij}, (i,j) \in S\}\]
(\(E\) requires \(\{(i,j):(A_{B,\kappa})_{ij} \ne 0\}=S \cup \{(j,i): (i,j) \in S\},\) while \(E^B\) additionally fixes \(\{(A_{B,\kappa})_{ij} \ne 0\}\)). These two families consist of pairwise disjoint events, with their complements contained in the union of \(\{max(A)>M\}\) and an event of small probability (see subsection~\ref{7.3}). Hence showing
\[\lim_{n \to \infty}{\mathbb{P}(||P_s+P_{B,\kappa}||>\max{(F(\theta),f(M))}+\epsilon \hspace{0.05cm} | \hspace{0.05cm} E(S,\kappa,M))}=0\]
\[\lim_{n \to \infty}{\mathbb{P}(\lambda_1(P_s+P_{B,\kappa})<\max{(F(\theta),f(max(A)))}-\epsilon \hspace{0.05cm} | \hspace{0.05cm} E^B(S,\kappa,M))}=0\]
uniformly in \(S \in \mathcal{S}\) (i.e., bounds involving solely \(M, m, n\)) suffices. For the sake of simplicity, denote conditioning on these events by \(*,**,\) respectively, and fix \(S.\) The key inequalities are for \(p \leq n^{1/7}, n \geq n_0,\)
\begin{equation}\label{eq20}
    s_1(\theta,p) \leq \mathbb{E}[tr(P_s^{2p})-tr(A_s^{2p})] \leq \theta^{2p}+2p \cdot s_1(\theta,p)+O(n^{-1}p^5s_1(\theta,p))+O(p(\max{(\theta,2)})^{2p}),
\end{equation}
\begin{equation}\label{eq10}
    \mathbb{E}_{*}[tr((P_s+P_{B,\kappa})^{2p})-tr(P_s^{2p})] \leq 2m \cdot s(M+\epsilon,p)+O(2m ||v||^2_{\infty} \cdot p^{12p}(c(M,\epsilon,\theta))^p)
\end{equation}
with \(s_1:(0,\infty)\times \mathbb{N} \to (0,\infty), \lim_{p \to \infty}{(s_1(\theta,p))^{1/p}}=F^2(\theta)\) (\(s_1(\theta,p)\) is given by (\ref{mainsum})) and \(\epsilon>0.\) Reasoning as in subsection~\ref{7.1} for \(P_1,P_2,\) \(||P_m||=||A_m||+o(1)=o(1),||P_{b,\kappa}||=||A_{b,\kappa}||+o(1) \leq \kappa+o(1)\) with high probability. Once (\ref{eq20}) and (\ref{eq10}) are justified (subsections~\ref{3.1} and \ref{3.2}), it is shown the variances of these trace differences are small for \(p \leq c(n)\), which entails
\[||P_s|| \xrightarrow[]{p} \max{(F(\theta),2)}, \hspace{0.5cm}  ||P_s+P_{B,\kappa}|| \leq \max{(F(\theta),f(M))}+\epsilon\]
with high probability (in this order), and the counterpart of (\ref{eq10}) for odd powers gives the result for the largest eigenvalues \(\lambda_1(P_s),\lambda_1(P_s+P_{B,\kappa})\) (subsection~\ref{3.3}). 
\par
As before, \(P_s,P_m,P_{b,\kappa},P_{B,\kappa}\) are obtained from the decomposition
\[p_{ij}=p_{ij}\chi_{|a_{ij}| \leq n^{1/4-\delta_1}}+p_{ij}\chi_{n^{1/4-\delta_1}<|a_{ij}| \leq n^{3/8+\delta_2}}+p_{ij}\chi_{n^{3/8+\delta_2}<|a_{ij}| \leq \kappa \sqrt{n} }+p_{ij}\chi_{\kappa \sqrt{n} \leq |a_{ij}|}.\]

\subsection{Edge Eigenspectrum of \(P_s\)}\label{3.1}

Begin with (\ref{eq20}): in light of \((P_s)_{ij}=(A_s)_{ij}+\theta v_iv_j \chi_{|a_{ij}| \leq n^{1/4-\delta_1}},\) use a similar analysis as for Theorem~\ref{th2}. C3 terms have a non-negative contribution that is at most \(\theta^{2p},\) there are no C1 summands, and it remains to consider the C2 ones. Keep the notation and terminology previously employed: a closer look at subsection \(2.1\) of \cite{oldpaper} reveals the bound in (\ref{traceineq}) (inequality \((15)\) therein) can be refined:

\begin{lemma}\label{lemma20}
Under the assumptions stated above (\ref{traceineq}) and \(L(n) \leq n^{\delta/2},\) (\ref{traceineq}) can be replaced for \(n \geq n(\epsilon), p \leq n^{\delta_4}\) by 
\[C_pn^{p+1} \leq \mathbb{E}[tr(B^{2p})] \leq C_pn^{p+1}+n^{p+1-2\delta}2^{2p}L^2(n)e^8+2^{2p}n^{p}L(n)e^8 \leq (1+\epsilon)C_pn^{p+1}.\] 
\end{lemma}

\begin{proof}
Recall (\ref{eq14}):
\[\sum_{e,k(e) \geq 2}{k(e)} \leq E+1+\sum_{k \geq 2}{kn_k},\]
which in step \(5\) (subsection \(2.1,\) \cite{oldpaper}) yields
\[\mathbb{E}[b_{\mathbf{i}}] \leq \prod_{e \in \mathbf{i}, k(e) \geq 2}{L(n)n^{\delta(2k(e)-4)}}=(L(n))^{E}n^{2\delta(\sum_{k(e) \geq 2}{k(e)}-2E)} \leq\]
\[\leq (L(n))^{E}n^{2\delta(1+\sum_{k \geq 2}{kn_k}-E)} \leq L(n)n^{2\delta\sum_{k \geq 2}{kn_k}},\]
since when \(E=0, \mathbb{E}[b_{\mathbf{i}}] \leq 1,\) and \(L(n) \leq n^{2\delta}\) for \(E \geq 1.\)
\par
Consider the upper bound. If \(E \geq 2,\) then a factor of \(n^{-2\delta}L(n)\) can be added to the bound in (\ref{traceineq}), yielding at most \(n^{p+1-2\delta}2^{2p}L^2(n)e^8;\) \(E=0\) gives \(C_pn^{p+1}\) (these are the elements of \(\mathcal{C}(p)\) and \(|\mathcal{C}(p)|=C_p\)); if \(E=1,\) then cycles with (\ref{eq14}) strict are included in the first case. Suppose equality holds: then \(i_0 \in N_{\mathbf{i}}(1).\) To see this, for each \(e, k(e) \geq 2,\) either one of its endpoints is marked at most once, in which case all its marked copies have the other endpoint as a right endpoint apart from at most one, or both are marked exactly once. If all edges are in the former category, the inequality is strict (\(1\) can be dropped), while the latter scenario can occur only if \(e=i_0i_1, i_0 \in N_{\mathbf{i}}(1):\) let \((u,v)=(i_t,i_{t+1})\) be the first apparition of \(e=uv, u, v  \in N_{\mathbf{i}}(1);\) then \(t=0\) (otherwise, \(u\) is marked at least twice, for a copy of \(e\) and \((i_{t-1},i_t)\)). Since \(i_0\) is marked, a factor of \(n\) can be dropped, yielding at most \(2^{2p}n^pL(n)e^8.\) Lastly, the lower bound ensues from the contribution of the elements in \(\mathcal{C}(p).\)   
\end{proof}

\par
Keep the notation from subsection~\ref{7.3} for C2 terms: anew solely cycles whose unerased edges can be arranged in pairwise disjoint even cycles contribute to the trace. Lemma~\ref{lemma20} entails each generates a factor of at most constant order (in \(p, \theta:\) no dependence on \(n\)) because all have at least one fixed vertex. Intuitively, cycles with \(r_i \ne l_{i+1}\) for some \(1 \leq i \leq s, l_{s+1}:=l_1\) should not contribute first order terms as one of the disjoint components has two vertices fixed and thus another factor of \(n\) could be dropped. Before making this statement quantitative, consider the case \(r_i=l_{i+1}, 1 \leq i \leq s\) with any two of the distinct segments sharing no edge (i.e., there are \(s\) disjoint even cycles): they contribute, up to a factor of \((1+\epsilon)^p\) and polynomials in \(p,\) plus an additive error of size \(n^{-1}(\max{(\theta,1)})^{2p}C(p),\)
\begin{equation}\label{mainsum}
    s_1(\theta,p)=\sum_{L_1, L_2, \hspace{0.05cm} ... \hspace{0.05cm},L_s \geq 1, L_1+...+L_s \leq p-s/2}{\binom{2p-2L_1-..-2L_s-1}{s-1}\theta^{2p-2L_1-..-2L_s}C_{L_1}C_{L_2}...C_{L_s}}:
\end{equation}
suppose these disjoint cycles have lengths \((2L_i)_{1\leq i \leq s};\) once \((L_i)_{1\leq i \leq s}\) are selected, choose the first vertex (at most \(2p\) possibilities) and assign the number of edges of the deleted segments, which gives the binomial coefficient (the upper bound \(p-s/2\) comes from each unerased segment containing at least an edge).
\par
Consider now the remaining cycles: suppose \(2l\) of its edges are not erased, form \(s\) segments, fix \(l,s,\) denote by \(\sigma(s,l)\) the overall contribution of such cycles for the erased clusters having fixed lengths and 
\[\Sigma(s,l)=\sigma(s,l)+\theta^{2p-2l}\sum_{L_1, L_2, \hspace{0.05cm} ... \hspace{0.05cm},L_s \geq 1, L_1+...+L_s=l}{C_{L_1}C_{L_2}...C_{L_s}}:\]
the difference \(\Sigma(s,l)-\sigma(s,l)\) accounts for the cycles described above for a given distribution of lengths of the erased edges. Note \(\Sigma(s,l)-\sigma(s,l)\) does not depend on the latter. If \(r_i \ne l_{i+1}\) for some \(1 \leq i \leq s,\) then glue the first such segment to the next one adjacent to it: this yields at most \(n^{-1} \cdot p^2 \cdot (2p)^3 \Sigma(s-1,l)\) because there are at most \(p^2\) configurations of erased clusters (two of the initial ones are merged), the even cycle containing this glued segment can drop a factor an \(n^{-1},\) and any preimage of this map has size at most \((2p)^3\) (choose where in the cycle to split a segment into two, the index of the first segment it is pulled to, and the first vertex of the configuration). Otherwise, \(r_i=l_{i+1}\) for all \(1 \leq i \leq s:\) if there are two segments sharing an edge, merge them, yielding a factor of \(n^{-1} \cdot n^{1/2-2\delta_1} \leq n^{-1/2}\)
and the overall contribution is \(p^2 \cdot (2p)^3 \cdot n^{-1/2}\Sigma(s-1,l).\) Putting these two situations together yields
\[\sigma(s,l) \leq 16p^5n^{-1}\Sigma(s-1,l)= 16p^5n^{-1}\sigma(s-1,l)+16p^5n^{-1}(\Sigma(s-1,l)-\sigma(s-1,l)).\]
Since \(\sigma(1,l) \leq 2\theta^{2p-2l}C_l,\)
\[\sigma(s,l) \leq (16p^5n^{-1})^{s-1} \theta^{2p-2l} C_l+32p^5n^{-1}(\Sigma(s,l)-\sigma(s,l))\]
(the difference is increasing in \(s\) and \(p^5n^{-1} \leq \frac{1}{2}\)), completing (\ref{eq20}) as \(\binom{2p-2l-1}{s-1}(16p^5n^{-1})^{s-1} \leq (32p^6n^{-1})^{s-1} \leq 1.\)
\par
Lemmas~\ref{lemma4}-\ref{lemma79} justify the claims about \(F\) made in the theorem statement and the ensuing remark, completing this subsection (\(l=\sum_{1 \leq i \leq s}{L_i}=px,s=py\) in Lemma~\ref{lemma79}).

\begin{lemma}\label{lemma4}
For \(l,s \in \mathbb{N}, l \geq s,\)
\[\sum_{l_1, l_2, \hspace{0.05cm} ... \hspace{0.05cm}, l_s \geq 1, l_1+...+l_s =l}{C_{l_1}C_{l_2}...C_{l_s}}=\sum_{l_1, l_2, \hspace{0.05cm} ... \hspace{0.05cm},l_{2s} \geq 0, l_1+...+l_{2s}=l-s}{C_{l_1}C_{l_2}...C_{l_{2s}}}.\]
\end{lemma}

\begin{proof}
Use the generating function of the Catalan numbers:
\[c(x)=\sum_{l \geq 0}{C_lx^l}=\frac{1-\sqrt{1-4x}}{2x},\]
for which \(c(x)=1+xc^2(x).\) The left-hand side term above is \(\frac{1}{l!} \cdot \frac{\partial^l{(c(x)-1)^s}}{\partial{x^l}}|_{x=0},\) and
Leibniz formula yields
\[\frac{1}{l!} \cdot \frac{\partial^l{(c(x)-1)^s}}{\partial{x^l}}=\frac{1}{l!} \cdot \frac{\partial^l{(x^sc^{2s}(x))}}{\partial{x^l}}=\frac{1}{l!} \sum_{0 \leq k \leq s}{\binom{l}{k} s(s-1)...(s-k+1)x^{s-k} \cdot  \frac{\partial^{l-k}{(c^{2s}(x))}}{\partial{x^{l-k}}}},\]
whereby the sum of interest is
\[\frac{1}{l!} \cdot \binom{l}{s} s! \frac{\partial^{l-s}{(c^{2s}(x))}}{\partial{x^{l-s}}}|_{x=0}=\frac{1}{(l-s)!} \cdot \frac{\partial^{l-s}{(c^{2s}(x))}}{\partial{x^{l-s}}}|_{x=0}.\]
\end{proof}

\begin{lemma}\label{lemma78}
Let \(s, l \in \mathbb{N}, r:\mathbb{N} \to (0,\infty),\)
\[\sigma(l,s)=\sum_{l_1, l_2, \hspace{0.05cm} ... \hspace{0.05cm},l_{s} \geq 0, l_1+...+l_{s}=l}{C_{l_1}C_{l_2}...C_{l_{s}}}, \hspace{0.5cm} r(l)=\sum_{i \geq l}{\frac{C_i}{4^i}}.\]
Then for \(r(0)=0,\)
\[\sigma(l,s) \leq 2^{-s}C_{l+s}\prod_{0 \leq i \leq s-1}{(1+\frac{r(i)}{2})},\]
and for \(l>s,\)
\[\sigma(l,s) \geq \frac{C_l}{4} \cdot ((\frac{5}{4})^s-1).\]
\end{lemma}

\begin{proof}
Begin with the upper bound: use induction on \(s\) and
\begin{equation}\label{recurr}
    \sigma(l,s)=\sum_{0 \leq i \leq l}{C_i\sigma(l-i,s-1)}.
\end{equation}
Clearly,
\[\sigma(l,1)=C_l=C_{l+1} \cdot \frac{l+2}{4l+2} \leq \frac{C_{l+1}}{2};\] 
for \(s \geq 2,\) the induction hypothesis and (\ref{recurr}) yield
\[\sigma(l,s) \leq 2^{-s+1}\prod_{0 \leq i \leq s-2}{(1+\frac{r(i)}{2})}\sum_{0 \leq i \leq l}{C_iC_{l-i+s-1}}=2^{-s+1}\prod_{0 \leq i \leq s-2}{(1+\frac{r(i)}{2})}(C_{l+s}-\sum_{l+1 \leq i \leq l+s-1}{C_iC_{l-i+s-1}}).\]
Since \(C_{l+1} \leq 4C_l,\) \(C_{l} \geq C_{l+k}4^{-k}\) for \(k \geq 0,\) from which 
\[\sigma(l,s) \leq 2^{-s+1}\prod_{0 \leq i \leq s-2}{(1+\frac{r(i)}{2})}(C_{l+s}-C_{l+s}\sum_{0 \leq i \leq s-2}{C_i4^{-i-1}})=\]
\[=2^{-s+1}C_{l+s}\prod_{0 \leq i \leq s-2}{(1+\frac{r(i)}{2})}(1-\frac{1}{2}+\frac{r(s-1)}{4})=2^{-s}C_{l+s}\prod_{0 \leq i \leq s-1}{(1+\frac{r(i)}{2})}\]
as \(\sum_{i \leq 0}{\frac{C_i}{4^i}}=c(\frac{1}{4})=2.\)
\par
For the lower bound, note 
\[\sigma(l,s) \geq \sum_{0 \leq k \leq s-1}{\binom{s}{k+1}(k+1)C_{l-k}} \geq C_l\sum_{0 \leq k \leq s-1}{\binom{s}{k+1}4^{-k}}=\frac{C_l}{4} \cdot ((\frac{5}{4})^s-1)\]
by taking tuples with one entry \(l-k,\) \(k\) one, \(s-1-k\) zero for \(0 \leq k \leq s-1.\)

\end{proof}

Lemmas~\ref{lemma4} and \ref{lemma78} yield that for any \(\epsilon>0,\)
\[\frac{C_{l-s}}{4} \cdot ((\frac{5}{4})^{2s}-1) \leq \sum_{l_1, l_2, \hspace{0.05cm} ... \hspace{0.05cm}, l_s \geq 1, l_1+...+l_s =l}{C_{l_1}C_{l_2}...C_{l_s}} \leq 2^{-2s}C_{l+s}(1+\epsilon)^{s}C(\epsilon)\]
as \(\lim_{t \to \infty}{r(t)}=0,\) where the lower bound holds solely for \(l>3s.\) The next result completes the stated claims.

\begin{lemma}\label{lemma79}
Let \(D_1=\{(x,y) \in [0,1]^3: y \leq x, y \leq 2-2x\},D_2=\{(x,y) \in [0,1]^3: y \leq x/3, y \leq 2-2x\}\), and \(\newline h_1:D_1 \to \mathbb{R}, h_2:D_2 \to \mathbb{R},\)
\[h_1(x,y)=\theta^{2-2x} \cdot \frac{(2-2x)^{2-2x}}{y^y \cdot (2-2x-y)^{2-2x-y}} \cdot 4^x,\]
\[h_2(x,y)=\theta^{2-2x} \cdot \frac{(2-2x)^{2-2x}}{y^y \cdot (2-2x-y)^{2-2x-y}} \cdot 4^{x-y} \cdot (\frac{5}{4})^{2y},\]
where \(0^0:=\lim_{x \to 0+}{x^x}=1.\) Then
\[\sup_{(x,y) \in D_1}{h_1(x,y)}=\max{(4,G_1^2(\theta))}, \hspace{0.5cm} \sup_{(x,y) \in D_2}{h_2(x,y)}=\max{(4,G_2^2(\theta)}),\]
where \(G_1,G_2:(0,\infty) \to (0,\infty),\) 
\[G_1(y)=y \cdot \frac{2-2x}{2-3x}, \hspace{0.5cm} x=x(y) \in (0,2/3), \hspace{0.5cm} \frac{(2-3x)^3}{x(1-x)^2}=y^2,\]
\[G_2(y)=y \cdot \frac{2-2x}{2-7x}, \hspace{0.5cm} x=x(y) \in (0,6/7), \hspace{0.5cm} \frac{(6-7x)^{7/3}}{x^{1/3}(1-x)^2}=\frac{36y^2}{5^{2/3}}.\]
Moreover,
\[\begin{cases}
    G_1(\theta)< f(\theta), & \theta < 1 \\
    G_1(\theta)> f(\theta), & \theta > 1 
    \end{cases}, \hspace{0.5cm} G_2(\theta) \leq 2 \Longleftrightarrow \theta \leq \frac{128}{89}, \hspace{0.5cm} G_2(\theta)<f(\theta).\]
\end{lemma}

\begin{proof}
Since \(\frac{1}{h_1} \cdot \frac{\partial{h_1}}{\partial{y}}=\log{\frac{2-2x-y}{y}},\) the maximum is attained at \(y=1-x\) if \(x \geq \frac{1}{2}\) or \(y=\min{(x,2-2x)}\) if \(x \leq \frac{1}{2}:\) 
\begin{equation}\label{max1}
    \sup_{x \geq \frac{1}{2}}{h_1(x,1-x)}=\max{(2\theta,4)},
\end{equation}
\begin{equation}\label{max2}
    \sup_{x \geq \frac{2}{3}}{h_1(x,2-2x)}=\max{(4^{2/3}\theta^{2/3},4)},
\end{equation}
while for \(x \leq \frac{2}{3},\)
\[h_1(x,x)=\theta^{2-2x} \cdot \frac{(2-2x)^{2-2x}}{x^x \cdot (2-3x)^{2-3x}} \cdot 4^x,\]
with the critical point \(x_0<\frac{2}{3},\)
\[\frac{(2-3x_0)^3}{x_0(1-x_0)^2}=\theta^2,\]
giving the maximum \(h_1(x_0,x_0)=\theta^2 \cdot (\frac{2-2x_0}{2-3x_0})^2\) (\(g(x)=\log{\frac{(2-3x)^3}{x(1-x)^2}}\) is decreasing in \(x \in (0,\frac{2}{3})\) as \[g'(x)=-\frac{9}{2-3x}-\frac{1}{x}+\frac{2}{1-x}=\frac{3x-5}{(1-x)(2-3x)}-\frac{1}{x}<0,\]
and \(\lim_{x \to 0+}{g(x)}=\infty, \lim_{x \to \frac{2}{3}}{g(x)}=-\infty\)). The suprema in (\ref{max1}) and (\ref{max2}) are at most \(\max{(\theta^2,4)}\) and \(\newline h_1(x_0,x_0) \geq \theta^2,\) whereby
\[\sup_{(x,y) \in D}{h_1(x,y)}=\max{(4,\theta^2 \cdot (\frac{2-2x_0}{2-3x_0})^2)}.\]
Let us now justify the stated properties for \(G_1:\) \(4 \geq \theta^2 \cdot (\frac{2-2x_0}{2-3x_0})^2\) is equivalent to \(\theta<2, x_0 \leq \frac{2-\theta}{3-\theta},\)
\[\frac{(2-3 \cdot \frac{2-\theta}{3-\theta})^3}{\frac{2-\theta}{3-\theta} \cdot (1-\frac{2-\theta}{3-\theta})^2}=\frac{\theta^3}{2-\theta} \leq \theta^2\]
or \(\theta \leq 1.\) If \(\theta>1,\) then \(G_1(\theta)> f(\theta)\) because \(\frac{2-2x_0}{2-3x_0}>1+\frac{1}{\theta^2}\) or \(x_0>\frac{2}{3+\theta^2}:=t, \frac{1}{\theta^2}=\frac{t}{2-3t}\) from \(t < \frac{1}{2},\)
\[\frac{(2-3t)^3}{t(1-t)^2\theta^2}=\frac{(2-3t)^2}{(1-t)^2} > 1.\]
\par
Proceed now with the second function: note
\[\frac{1}{h_2} \cdot \frac{\partial{h_2}}{\partial{x}}=2\log{\frac{2-2x-y}{\theta(1-x)}}, \hspace{0.5cm} \frac{1}{h_2} \cdot \frac{\partial{h_2}}{\partial{y}}=\log{\frac{(2-2x-y)(5/8)^{2}}{y}}.\]
The maximum is attained on the boundary, i.e., \(\{y=0\} \cup \{y=x/3\} \cup \{y=2-2x\},\) or at a point at which a partial derivative vanishes. 
\par
\underline{Case 1:} \(y=0:\max_{x \in [0,1]}{h_2(x,0)}=\max{(\theta^2,4)}.\)
\par
\underline{Case 2:} \(y=x/3:\)
\[h_2(x,x/3)=\theta^{2-2x} \cdot \frac{(2-2x)^{2-2x}}{(x/3)^{x/3}(2-7x/3)^{2-7x/3}} \cdot 5^{2x/3}  :=h_3(x).\]
As before, the maximum is
\[\theta^2 \cdot \frac{(2-2x_1)^2}{(2-7x_1/3)^2},\] 
attained at the critical point \(x_1 \in (0,\frac{6}{7})\) with
\[\frac{(6-7x_1)^{7/3}}{x_1^{1/3}(1-x_1)^2}=\frac{36\theta^2}{5^{2/3}}.\]
\par
\underline{Case 3:} \(y=2-2x, x>\frac{6}{7}:\) for \(\epsilon>0\) sufficiently small, \(\frac{\partial{h_2}}{\partial{y}}<0\) for \(y \geq 2-2x-\epsilon,\) and thus, the function is decreasing in \(y,\) implying the maximum cannot have \(y=2-2x\) (the restraint is \(y \leq 2-2x\) for \(x>\frac{6}{7}\)).
\par
\underline{Case 4:} \(\frac{\partial{h_2}}{\partial{x}}=0.\)
This forces \(\theta<2, (1-x)(2-\theta)=y.\) Then
\[h_2(x,y)=\theta^2 \cdot (\frac{2-2x}{2-2x-y})^2 \cdot (\frac{2-2x-y}{y})^y \cdot (\frac{5}{8})^{2y}= 4 \cdot (\frac{\theta}{2-\theta} \cdot \frac{25}{64})^y,\]
for which the supremum is obtained at \(y=0\) or \(y=x/3\) (\(y \leq x/3\) is the remaining constraint), covered by cases \(1\) and \(2.\)
\par
\underline{Case 5:} \(\frac{\partial{h_2}}{\partial{y}}=0.\)
This is tantamount to \(2-2x=y(1+\frac{64}{25}),\) and 
\[h_2(x,y)=\theta^{2-2x} \cdot (\frac{2-2x}{2-2x-y})^{2-2x} \cdot 4^{x}=\theta^{2-2x} \cdot (1+\frac{25}{64})^{2-2x} \cdot 4^{x}.\]
The condition on \(x\) comes from \(y \leq x/3,\) entailing the maximum is attained at \(x=3y\) or \(x=1,y=0,\) both already considered. 
\par
The case-by-case analysis above gives
\[\sup_{(x,y) \in D}{h_2(x,y)}=\max{(4,\theta^2 \cdot (\frac{2-2x_1}{2-7x_1/3})^2)}.\] 
Turn now to \(G_2:\) reasoning as for \(h_1\) yields \(4 \geq \theta^2 \cdot \frac{(2-2x_1)^2}{(2-7x_1/3)^2}\) equivalent to \(\theta<2, x_1 \leq \frac{2-\theta}{7/3-\theta},\) becoming
\[\frac{\theta^{7/3}}{(2-\theta)^{1/3}(1/3)^2} \leq \frac{36\theta^2}{5^{2/3}},\]
or \(\frac{\theta}{2-\theta} \leq \frac{64}{25}, \theta \leq \frac{128}{89}.\)
In particular,
\[\sup_{(x,y) \in D}{h_2(x,y)}>4 \Longleftrightarrow \theta>\frac{128}{89},\]
and \(G_2(\theta)<f(\theta)\) for all \(\theta>0:\) \(\frac{2-2x_1}{2-7x_1/3}< 1+\frac{1}{\theta^2}, x_1< \frac{6}{7+\theta^2}\) from
\[\frac{(6\theta^2)^{7/3}}{6^{1/3}(1+\theta^2)^2}=\frac{36\theta^{14/3}}{(1+\theta^2)^2} < \frac{36\theta^2}{5^{2/3}},\]
equivalent to \(\frac{\theta^{4}}{(1+\theta^2)^3} < \frac{1}{5},\) which holds because \((1+x)^3-5x^2=1+x((x-1)^2+2)>0\) when \(x=\theta^2.\)
\end{proof}

\subsection{\(P_s+P_{B,\kappa}\) versus \(P_s\)}\label{3.2}

Having completed the justification of (\ref{eq20}), proceed with (\ref{eq10}). The edges of the cycles in the trace difference belong either to \(P_{B,\kappa}\) or to \(P_s;\) let \(t\) be the number of clusters of the former type and \(2p-l\) their number of edges; by freezing them, the contributions again can be split in C1, C2, and C3 terms. C3 generate \(|\theta^{2p-l}\prod{(v_{l_i}v_{r_i})}| \leq \theta^{2p-l} \cdot ||v||_{\infty}^{2t},\) and overall these give at most
\[2p\sum_{1 \leq l \leq 2p-1, l \geq t}{\binom{2p-l-1}{t-1}\binom{l-1}{t-1}(M+\epsilon)^{2p-l} \cdot \theta^l||v||_{\infty}^{2t}(2m)^t}:\]
let \(l=l_1+l_2+...+l_t,\) where \(l_i\) is the length of the \(i^{th}\) remaining segment; there are \(\binom{(l-t)+t-1}{t-1}\) such tuples, for each, \(\binom{(2p-l-t)+t-1}{t-1}=\binom{2p-l-1}{t-1}\) ways to choose the lengths of the clusters of edges belonging to \(P_{B,\kappa},\) and once these are fixed, there are at most \(2m\) possibilities per cluster as each is predetermined by its first edge and its length (recall property \((c)\) of \(S\)); the expectation is at most \((M+\epsilon)^{2p-l} \cdot \theta^l||v||_{\infty}^{2t}\) since the moments of entries belonging to \(P_{B,\kappa}\) are controlled by powers of \(M+\epsilon.\) Because \(t \geq 1,\) the sum is of order \(2m \cdot ||v||_{\infty} \cdot (c(M,\epsilon,\theta))^p\) for \(2m \cdot ||v||_{\infty} \leq 2\sqrt{||v||_{\infty}} \leq \frac{1}{2}.\)
\par
Continue with the C1 terms: the segments left can be glued into a collection of disjoint even cycles (otherwise the initial expectation is zero). If \(r_i \ne l_{i+1},\) then a power of \(n\) can be saved (one of the cycles will have two fixed vertices). Otherwise, \(r_i=l_{i+1},\) implying all the clusters contain the same undirected edge. Hence the cycle itself is even and the contribution of these terms is at most \(2m \cdot s(M+\epsilon,p).\)
\par
Lastly, take the C2 summands: supoose \(s \geq 1\) segments of edges belonging to \(P_s\) are erased and a total of \(L\) edges; again these erased components generate at most \(\theta^{L}\prod{(v_{l_i}v_{r_i})}.\) Look at the collection of segments left; reasoning as before, there is a saving of a power of \(n\) unless these are pairwise edge-disjoint even cycles. In this last case, there is only one undirected edge belonging to \(P_{B,\kappa}\) among the \(2p-l\) ones, and the overall contribution is \(O(2m ||v||^2_{\infty} \cdot p^{12p}(c(M,\epsilon,\theta))^p)\) for \(2m \cdot ||v||^2_{\infty} \leq 2||v||_{\infty} \leq \frac{1}{2}\)
arguing as in subsection~\ref{7.3}. 

\subsection{Edge Eigenspectrum of \(P_s+P_{B,\kappa}\)}\label{3.3}

\begin{equation}\label{eq3}
    Var_{**}[tr((P_s+P_{B,\kappa})^{2p+1})-tr(P_s^{2p+1})]= O((||v||^2_{\infty}+n^{-1/2})[2m \cdot s(M+\epsilon,4p) +p^{12p}(c(M,\theta,\epsilon))^{p}])
\end{equation}
\begin{equation}\label{eq4}
    \lambda_1(P_s), ||P_s|| \xrightarrow[]{p} \max{(F(\theta),2)}, 
\end{equation}
for all \(\epsilon>0.\)
\par
Begin with (\ref{eq3}): since the components are cycles, the variance is a sum over pairs of cycles \(\mathbf{i},\mathbf{j}\); if they share no edge belonging to \(P_s,\) then they are independent. Else, they can be arranged in a cycle of length \(2(2p+1)\). As in \cite{oldpaper} (see also end of subsection~\ref{0.1}), this edge can be erased at a cost of order \(\theta ||v||^2_{\infty}+n^{2(1/4-\delta_1)}n^{-1}.\) (\ref{eq10}) completes the claim.
\par
Continue with (\ref{eq4}), for which
\[Var[tr(P_s^{2p}))]=\sum_{(\mathbf{i},\mathbf{j})}{(\mathbb{E}[a_{\mathbf{i}}a_{\mathbf{j}}]-\mathbb{E}[a_{\mathbf{i}}] \cdot \mathbb{E}[a_{\mathbf{j}}])},\]
where the cycles \(\mathbf{i},\mathbf{j}\) have length \(2p\) and edges either erased or belonging to \(A_s.\) Independence implies only pairs with \(\mathbf{i},\mathbf{j}\) sharing an edge \(e\) can make nonzero contributions, and 
as in the unperturbed case, glue them into a cycle \(\mathcal{P}\) of length \(2 \cdot 2p=4p\) along the first copy of \(e\) (in \(\mathbf{i}\) and \(\mathbf{j},\) in this order). Let \(l_1,l_2,(l_3,l_4)\) be the numbers of copies of \(e=k_1k_2\) belonging to \(A_s\) (erased) in \(\mathbf{i},\mathbf{j},\) respectively. If \(l_1=l_2=0,\) then the difference is zero as \(\chi_{E} \cdot \chi_{E}=\chi_{E}.\) If \(l_1,l_2 \geq 1,\) then two copies of \(e\) can be deleted, generating a factor of \(n^{2(1/4-\delta_1)}n^{-1}=n^{-1/2-2\delta_1},\) and what remains is a cycle of length \(4p-2\) contributing \(\mathbb{E}[tr(P^{4p-2}_s)].\) Else, suppose \(l_1>0,l_2=0,\) and consequently \(l_4>0.\) Recall the analysis of \(\mathbb{E}[tr(P^{2p}_s)]:\) the change of summation freezes the erased edges and considers the remaining segments which can be arranged into a collection of disjoint even cycles (otherwise, their expectation is zero). For such cycles, at least one vertex in one of the cycles is fixed (i.e., the first copy of \(e\) from \(\mathbf{i}\) belonging to \(A_s\)), generating a factor of \(n^{-1}.\) Hence
\begin{equation}\label{varps}
    Var[tr(P_s^{2p}))] \leq n^{-1/2} \cdot 4p \cdot \mathbb{E}[tr(P^{4p-2}_s)]+n^{-1}\cdot 4p \cdot \mathbb{E}[tr(P^{4p}_s)]
\end{equation}
(the factors of \(p\) come from the preimage of a cycle of length \(2p\) into a pair \((\mathbf{i},\mathbf{j}):\) for details, see subsection \(3.2\) in \cite{oldpaper}). One more ingredient is needed for (\ref{eq4}): \(P_s-A_s\) has all eigenvalues are small apart from its largest.

\begin{lemma}\label{lem1}
Let \(T=(\theta v_iv_j\chi_{|a_{ij}| \leq n^{1/4-\delta_1}}).\) Then with probability at most \(O(n^{-\delta_1}),\)
\[\lambda_{1}(T) \geq \frac{\theta}{2}, \hspace{0.5cm} \lambda_{(2)}(T) \geq n^{-1/2+3\delta_1},\]
where \(||T||=\lambda_{(1)}(T) \geq \lambda_{(2)}(T) \geq ... \geq \lambda_{(n)}(T)\) are the ordered statistics of \((|\lambda_i(T)|)_{1 \leq i \leq n}.\)
\end{lemma}

\begin{proof}
For \(n\) sufficiently large, Chebyshev's inequality yields
\(\lambda_1(T) \leq \frac{\theta}{2}\) with probability at most \(O(n^{-1+4\delta_1})\) since
\[\mathbb{E}[v^TTv]=\theta\mathbb{P}(|a_{11}| \leq n^{1/4-\delta_1}) \geq \frac{3\theta}{4},\]
\[Var(v^TTv) \leq 4\theta^2\mathbb{P}(|a_{11}| \leq n^{1/4-\delta_1})(1-\mathbb{P}(|a_{11}| \leq n^{1/4-\delta_1}))\sum_{i,j}{v_i^2v_j^2} \leq 4\theta^2\mathbb{P}(|a_{11}| > n^{1/4-\delta_1}) \leq 8c\theta^2n^{-1+4\delta_1}.\]
\par
Arguing as for \(P_1\) and \(P_2\) in the end of subsection~\ref{7.1},
\[\mathbb{E}[tr(T^{2p})]=\theta^{2p}+O(\theta^{2p}pn^{-1+4\delta_1})\]
since \(1 \geq (\mathbb{P}(|a_{11}| \leq n^{1/4-\delta_1}))^{p} \geq (1-4cn^{-1+4\delta_1})^p \geq 1-4cpn^{-1+4\delta_1},\) with the same rationale providing
\[\mathbb{E}[tr(T^{2p})tr(T^2)-tr(T^{2p+2})]=O(\theta^{2p+2}pn^{-1+4\delta_1}).\]
Then when \(p=o(n^{1-4\delta_1}),\) with probability at most \(O(n^{-\delta_1}),\)
\[\lambda^2_{(2)}(T) \geq 2^{2p-2}pn^{-1+5\delta_1}\]
using
\[tr(T^{2p})tr(T^2)-tr(T^{2p+2})=\sum_{1 \leq i,j \leq n, i \ne j}{\lambda^2_i(T)\lambda^{2p}_j(T)} \geq ||T||^{2p} \cdot \lambda^2_{(2)}(T) \geq (\frac{\theta}{2})^{2p}\lambda^2_{(2)}(T).\]
The result ensues by choosing \(p=1.\)

\end{proof}

\par
Consider now (\ref{eq4}): Markov's inequality in (\ref{eq20}) for \(p=\lfloor \frac{4\log{n}}{\epsilon} \rfloor\) together with Lemma~\ref{lemma20} yield:
\[\mathbb{P}(||P_s|| \geq \max{(F(\theta),2)}+\epsilon) \leq (2^{2p+1}n+\theta^{2p}+2p \cdot s_1(\theta,p)+O(n^{-1}p^5s_1(\theta,p))+O(p(\max{(\theta,2)})^{2p}))/(\max{(F(\theta),2)}+\epsilon)^{2p}=o(1)\]
since \(F(\theta) \geq \theta, e^{x} \leq 1+2x, x \in [0,\log{2}],\)
\[2^{2p+1}n(2+\epsilon)^{-2p} \leq 2n\exp(-p\epsilon/2)=o(1).\]
For the lower bound, note that \(\lambda_1(P_s) \geq \lambda_1(A_s)+\lambda_n(P_s-A_s) \geq \lambda_1(A_s)-n^{-1/3}:\)
Lemma~\ref{lem1} renders \[\lambda_1(P_s-A_s) \geq \theta-\epsilon, \hspace{0.5cm} tr((P_s-A_s)^2) \leq \theta^2+\epsilon,\] 
whereby \(|\lambda_n(P_s-A_s)|=O(\epsilon)<||P_s-A_s||.\) This yields \(||P_s|| \geq 2-\epsilon\) for \(\epsilon>0\) with high probability, providing the result if \(F(\theta)<2.\) Suppose next \(F(\theta)>2+\delta,\) and take \(0<\epsilon<C^2(\delta)/4, p=\lfloor \frac{\log{n}}{\sqrt{\epsilon}} \rfloor:\)  
\[\mathbb{P}(||P_s|| \leq \max{(F(\theta),2)}-\epsilon) \leq \mathbb{P}(tr(P^{4p}_s) \leq (\max{(F(\theta),2)}-\epsilon)^{2p} \cdot tr(P^{2p}_s)) \leq\]
\[\leq \mathbb{P}(tr(P^{2p}_s) \geq (\max{(F(\theta),2)}+\epsilon)^{2p})+\mathbb{P}(tr(P^{4p}_s) \leq ((\max{(F(\theta),2)})^2-\epsilon^2)^{2p}) \leq\]
\[\leq 2n(\frac{2}{\max{(F(\theta),2)+\epsilon})})^{2p}+4p^5(\frac{\max{(F(\theta),2)}+\epsilon/2}{\max{(F(\theta),2)}+\epsilon})^{2p}+\mathbb{P}(tr(P^{4p}_s) \leq \mathbb{E}[tr(P^{4p}_s)]/2) \leq\]
\[\leq 2ne^{-C(\delta)p}+4p^5(1-\epsilon/4)^{2p}+4Var[tr(P^{4p}_s)]/(\mathbb{E}[tr(P^{4p}_s)])^2 =o(1)\]
because
\(n^{-1/2}(\max{(F(\theta),2)}+\epsilon)^{4p}/(\max{(F(\theta),2)}-\epsilon)^{4p} \leq n^{-1/2}e^{4\epsilon p}=o(1).\)
This completes the proof of (\ref{eq4}). To show the same result holds for \(\lambda_1(P_s),\) solely a lower bound is needed: again if \(F(\theta) \leq 2,\) then the result follows; else, use the odd moment trace inequality
\begin{equation}\label{eq200}
    \mathbb{E}[tr(P_s^{2p+1})] \geq s_1(\theta,p+1/2),
\end{equation}
which ensues from the justification of (\ref{eq20}) and the variance bound (\ref{eq4}) (again, \(2p+1\) can be replaced by \(2p\)):
\[\mathbb{P}(\lambda_1(P_s) \leq \max{(F(\theta),2)}-\epsilon) \leq \mathbb{P}(tr(P^{4p+1}_s) \leq (\max{(F(\theta),2)}-\epsilon)^{2p+1} \cdot tr(P^{2p}_s)) \leq\]
\[\leq \mathbb{P}(tr(P^{2p}_s) \geq (\max{(F(\theta),2)}+\epsilon)^{2p})+\mathbb{P}(tr(P^{4p+1}_s) \leq (\max{(F(\theta),2)}-\epsilon) \cdot ((\max{(F(\theta),2)})^2-\epsilon^2)^{2p}) \leq\]
\[\leq o(1)+\mathbb{P}(tr(P^{4p+1}_s) \leq \mathbb{E}[tr(P^{4p+1}_s)]/2) \leq o(1)+4Var[tr(P^{4p+1}_s)]/(\mathbb{E}[tr(P^{4p+1}_s)])^2 =o(1).\]
\par
Having completed the proofs of (\ref{eq3}) and (\ref{eq4}), reason as in the unperturbed case. To show the result for \(||P_s+P_{B,\kappa}||,\) both upper and lower bounds follow from \(P_{B,\kappa}\) being sparse, Lemmas~\ref{lemma7old} and \ref{lemma9old}, after which the odd moment analog of (\ref{eq10}):
\[\mathbb{E}_{**}[tr((P_s+P_{B,\kappa})^{2p+1})-tr(P_s^{2p+1})]=O((c(M,\theta,\epsilon))^{p}||v||_{\infty})\]
can be employed to deduce the result for \(\lambda_1(P_s+P_{B,\kappa}).\)

\section{Localized Eigenvector with \(\alpha=4\)}\label{sect2}

Suppose without loss of generality \(|v_i|=v_{(i)}, 1 \leq i \leq n;\) take \(\delta_1, \delta_2, \delta_3, \kappa>0\) as in section~\ref{sect3}, \(\delta_0>0\) such that \(\newline |v_{k}| \geq \delta_0, n \geq n_0,\) \(S(k)=\{(i,j):1 \leq i,j \leq k\}, (m_n) \subset \mathbb{N}, \lim_{n \to \infty}{m_n}=\infty,\) and decompose \(P\) as
\[P=P_k+P_s+P_m+P_{b,\kappa}+P_{B,\kappa},\]
\[(P_k)_{ij}=\begin{cases}
    \frac{1}{\sqrt{n}} a_{ij}, &  (i,j) \in S(k) \\
    0, & else
    \end{cases},\]
\[P_m=(p_{ij}\chi_{n^{1/4-\delta_1}<|a_{ij}| \leq n^{3/8+\delta_2}, (i,j) \not \in S(k)}), \hspace{0.1cm} P_{b,\kappa}=(p_{ij}\chi_{n^{3/8+\delta_2}<|a_{ij}| \leq \kappa \sqrt{n}, (i,j) \not \in S(k)}),\] 
\[P_s+P_{B,\kappa}=(p_{ij}\chi_{|a_{ij}| \leq n^{1/4-\delta_1}, (i,j) \not \in S(k)})+(p_{ij}\chi_{\kappa \sqrt{n} \leq |a_{ij}|,(i,j) \not \in S(k)}),\]
with 
\[(P_s)_{ij}=\begin{cases}
    0, &  (i,j) \in S(k) \\
    p_{ij}\chi_{|a_{ij}| \leq n^{1/4-\delta_1}}, & else
    \end{cases}, \hspace{0.5cm} (P_{B,\kappa})_{ij}=\begin{cases}
    \theta v_iv_j, &  (i,j) \in S(k) \\
    p_{ij}\chi_{|a_{ij}|>\kappa \sqrt{n}}, & else
    \end{cases}.\]
\(P_k\) is negligible because \(k \leq \delta_0^{-2},\) from which
\[||P_k||^2 \leq n^{-1}\max_{1 \leq i \leq k}{\sum_{1 \leq j \leq k}{a^2_{ij}}} \leq n^{-1} \cdot k(\max_{1 \leq i,j \leq k}{|a_{ij}|})^2=o_p(n^{-1/2}).\]
 Let \(\mathcal{S}=\mathcal{S}_{n,m}, E(S,\kappa,M), E^{B}(S,\kappa,M)\) be defined as in section~\ref{sect3}. In what follows, \(S \in \mathcal{S}\) is again fixed, and Theorem~\ref{th1} ensues from
\begin{equation}\label{easy10}
    ||P_m||=o_p(1), \hspace{0.5cm} ||P_{b,\kappa}|| \leq 2\kappa, \hspace{0.5cm} \lambda_1(P_s), ||P_s|| \xrightarrow[]{p} 2,
\end{equation}
\begin{equation}\label{easy12}
    ||P_s+P_{B,\kappa}||-\max{(f(\theta),f(max(A)))} \xrightarrow[]{p} 0,
\end{equation}
\begin{equation}\label{v2}
    \mathbb{E}_{**}[tr((P_s+P_{B,\kappa})^{2p+1})-tr(P_s^{2p+1})] \geq (1-o(1))^ps_2(\theta,p+1/2)-C(p)n^{-\delta_1}
\end{equation}
\begin{equation}\label{v3}
    Var_{**}(tr((P_s+P_{B,\kappa})^{2p+1})-tr(P_s^{2p+1}))=O(n^{-\delta_1}(C(\theta,M))^{p}),
\end{equation}
with the inequalities holding with high probability and the last two results uniform in \(S \in \mathcal{S}.\)
\par
Let us see first how these entail the theorem. By letting \(\kappa \to 0,\) (\ref{easy10}) implies it suffices to justify the convergence for \(\lambda_1(P_s+P_{B,\kappa});\) moreover, a lower bound suffices in light of (\ref{easy12}). For any \(\epsilon>0,\) with high probability
\begin{equation}\label{low13}
    \lambda_1(P_s+P_{B,\kappa}) \geq 2-\epsilon
\end{equation}
because \(\lambda_1(P_s+P_{B,\kappa}) \geq \lambda_1(P_s)+\lambda_n(P_{B,\kappa}),\) \(P_{B,\kappa}\) has at most \(2m\) nonzero entries together with (\ref{easy10}). The result therefore follows for \(max(A),\theta \leq 1.\) Otherwise, when \(max(A),\theta>1+\delta\) for \(\delta>0,\) if
\[\lambda_1(P_s+P_{B,\kappa})<\max{(f(\theta),f(max(A)))}-2\epsilon, ||P_s+P_{B,\kappa}|| \geq \max{(f(\theta),f(max(A)))}-\epsilon, ||P_s|| \leq  2+\epsilon,\] then Lemma~\ref{lemma9old} renders for \(\epsilon \leq \epsilon(\delta),m \leq c(\delta,\epsilon)\log{\log{p}},\)
\[tr((P_s+P_{B,\kappa})^{2p+1})-tr(P_s^{2p+1}) \leq 2m \cdot (\max{(f(\theta),f(max(A)))}-2\epsilon)^{2p+1}-(\max{(f(\theta),f(max(A)))}-\epsilon)^{2p+1}+\]
\[+3m \cdot  (2+\epsilon)^{2p+1} \leq -(\max{(f(\theta),f(max(A)))}-\epsilon)^{2p+1}/2;\]
as the right-hand side of (\ref{v2}) is at least \(1,\) this bound, (\ref{v3}), and Chebyshev's inequality give the desired claim.
\par
Since
\[P_{b,\kappa}=\frac{1}{\sqrt{n}}(a_{ij}\chi_{n^{3/8+\delta_2}<|a_{ij}| \leq \kappa \sqrt{n}})+(\theta v_iv_j\chi_{n^{3/8+\delta_2}<|a_{ij}| \leq \kappa \sqrt{n}}):=A_{b,\kappa}+P_3,\]
\[P_m=\frac{1}{\sqrt{n}}(a_{ij}\chi_{n^{1/4-\delta_1}<|a_{ij}| \leq n^{3/8+\delta_2}})+(\theta v_iv_j\chi_{n^{1/4-\delta_1}<|a_{ij}| \leq n^{3/8+\delta_2}}):=A_m+P_4,\]
(\ref{easy10}) ensues because \(||A_m||=o_p(1),||A_{b,\kappa}|| \leq \kappa, ||P_3||=o_p(1), ||P_4||=o_p(1),||P_5||=o_p(1):\) the first two have already been justified in section~\ref{sect3}, and for the last reason as in subsection~\ref{7.1}; for \(P_s,\) use (\ref{eq4}) (here \(\theta\) is replaced by \(\theta\sum_{i>k}{v^2_i}\)).
\par
Subsection~\ref{2.1} presents the proof of (\ref{easy12}), based on the analogues of 
(\ref{v2}) and (\ref{v3}) for even powers:
\begin{equation}\label{v22}
    \mathbb{E}_{*}[tr((P_s+P_{B,\kappa})^{2p})-tr(P_s^{2p})] \leq 2p \cdot s_2(\theta,p)+\theta^{2p}+2m \cdot s(M,p)+C(p)n^{-\delta_1},
\end{equation}
\begin{equation}\label{v222}
    \mathbb{E}_{**}[tr((P_s+P_{B,\kappa})^{2p})-tr(P_s^{2p})] \geq  s_2(\theta\sum_{1 \leq i \leq k}{v^2_i},p)-C(p)n^{-\delta_1},
\end{equation}
\begin{equation}\label{v33}
    Var_{*}(tr((P_s+P_{B,\kappa})^{2p})-tr(P_s^{2p}))=O(n^{-\delta_1}(C(\theta,M))^{p}),
\end{equation}
\begin{equation}\label{v333}
    Var_{**}(tr((P_s+P_{B,\kappa})^{2p})-tr(P_s^{2p}))=O(n^{-\delta_1}(C(\theta,M))^{p}),
\end{equation}
where \(s_2\) is defined by (\ref{longsum}). These identities entail the claimed convergence by reasoning as in subsection~\ref{3.3} and employing \(\lambda_1(P_s), ||P_s|| \xrightarrow[]{p} 2\) together with Lemmas~\ref{lemma7old} and \ref{lemma9old} (\(m=m_n, p\) can diverge arbitrarily slow). Since expectations of trace differences of even powers can be used to bound their variances (recall the gluing underlying (\ref{eq3}) described in subsection~\ref{0.1}), (\ref{v22}) and (\ref{v222}) suffice (subsection~\ref{2.1}). Finally, a slight adjustment of their justification renders (\ref{v2}) and (\ref{v3}), completing thus the proof of Theorem~\ref{th1}.

\subsection{Trace Difference}\label{2.1}

Under the conditioning \(*\) and the replacement of \(a_{ij},1 \leq i,j \leq k\) by zeros, \(P_{B,\kappa}\) differs from its counterpart \(A_{B,\kappa}\) in that \((P_{B,\kappa})_{ij}=\theta v_iv_j, 1 \leq i,j \leq k\) (whereas \((A_{B,\kappa})_{ij}=0, 1 \leq i,j \leq k\)). Focus on (\ref{v22}) first, as it will become apparent that (\ref{v222}) can be justified in a similar vein. 
\par
Distinct elements of \(S(k)\) can share an entry, whereas this is prohibited for \(S\) by its definition: this leads to more first order contributors to the expectation in the current situation than in the unperturbed case. Steps \(1'-6'\) described in subsection~\ref{0.1} still allow counting them. Lemma~\ref{lemma3old} continues to hold with a slight modification: cycles which become elements of \(\mathcal{C}(l)\) after zooming out the clusters of edges belonging to \(S(k)\) are also of type \((III)\) (i.e., map \(\mathbf{i}\) to \(\mathcal{C}(l)\) by replacing these segments with a vertex \(\rho;\) what results is an element of \(\mathcal{C}(l)\) since otherwise there will be some decay in such cycles, which would make them negligible: see the proof of this lemma in \cite{oldpaper}). Apart from the cycles coming from \(S,\) yielding at most \(2m \cdot s_1(M,p)+C(p)n^{-\delta_1},\) there are the cycles containing either only edges in \(S(k)\) (these contribute at most \(\theta^{2p}\)) or of type \((III)\) with at least one edge in \(S(k)\) and one belonging to \(P_s.\) Their overall contribution is, up to a factor of \(\sum_{1 \leq i \leq k}{v^2_i},\)
\begin{equation}\label{longsum}
    \sum_{1 \leq l \leq p-1, 1 \leq t \leq l+1,0 \leq q \leq t-1}{(\theta \sum_{1 \leq i \leq k}{v^2_i})^{2p-2l}\binom{t}{q+1} \binom{2p-2l-1}{q}b_{l,t}}:=s_2(\theta\sum_{1 \leq i \leq k}{v^2_i},p).
\end{equation}
To justify this, take \(\mathbf{i},\) an even cycle of type \((III),\) length \(2p\) with \(2p-2l\) edges belonging to \(S(k),\) map it to a cycle \(\mathbf{i}_c \in \mathcal{C}(l),\) denote by \(\rho \in V(\mathbf{i}_c)\) the vertex corresponding to the suppressed \(2p-2l\) edges, \(t\) its multiplicity in \(\mathbf{i}_c,\) and suppose \(q+1\) of these apparitions hide clusters of edges in \(S(k).\) The constraints on \(l, t, q\) are evident, and what remains is counting the preimages for a given triple \((l,t,q).\) For each maximal group of consecutive edges in \(S(k),\) solely its first and last vertices interact with \(\mathbf{i}_c\) and influence its structure; once these are fixed, establish the sizes of the \(q+1\) clusters: this can be done in \(\binom{(2p-2l-q-1)+(q+1)-1}{(q+1)-1}=\binom{2p-2l-1}{q}\) ways (recall (\ref{atl})) and there is a bijection between the tuples of interest and \(\{(x_1,\hspace{0.05cm}  ... \hspace{0.05cm} , x_{q+1}): x_1+...+x_{q+1}=2(p-l)-q-1, x_i \geq 0\}.\) The power of \(\theta\) is clear, \(\binom{t}{q+1}\) counts the number of sets of \(q+1\) copies of \(\rho\) pulled to clusters, while \(b_{l,t}\) gives the number of such vertices \(\rho.\) It remains to show that for each cluster configuration, its contribution is in \([(\sum_{1 \leq i \leq k}{v^2_i})^{q+1},(\sum_{1 \leq i \leq k}{v^2_i})^{q}]:\) note there are \(2p-2l-q-1\) hidden vertices (i.e., \(|\{s, 1 \leq s \leq 2p-1, (i_{s-1},i_s),(i_{s},i_{s+1}) \in \mathbf{i}''\}|=2p-2l-q-1\)), which generate a factor of \((\sum_{1 \leq i \leq k}{v^2_i})^{2p-2l-q-1}\) as they are arbitrary elements of \(s(k):=\{1,2, \hspace{0.05cm} ... \hspace{0.05cm}, k\}.\)
\par
\underline{Case 1:} \(\rho\) is not the first vertex of \(\mathbf{i}_c.\) 
Let \((s_0,\rho,s_1), \hspace{0.05cm} ... \hspace{0.05cm}, (s_{2t-2},\rho,s_{2t-1})\) be the apparitions of \(\rho\) in \(\mathbf{i}_c\) ordered increasingly. If \(t=1,\) then \(s_0=s_1,\) and the preimage of \(\rho\) is \((v_0,v_1, \hspace{0.05cm} ... \hspace{0.05cm},v_{2p-2l-1},v_0)\)
with \(v_0, v_1, \hspace{0.05cm} ... \hspace{0.05cm},v_{2p-2l-1} \in s(k),\) yielding \(\sum_{1 \leq i \leq k}{v^2_i}\) (\(v_0=v_{2p-2l}\) because otherwise \(\mathbf{i}\) is odd). 
\par
Else, \((s_2,\rho),(s_4,\rho), \hspace{0.05cm} ... \hspace{0.05cm},(s_{2t-2},\rho),(\rho,s_{2t-1})\) are unmarked, and so \((\rho,s_1),(\rho,s_3), \hspace{0.05cm} ... \hspace{0.05cm} ,(\rho,s_{2t-3})\) are marked. Consider again the left- and right-most vertices of the clusters in \(\mathbf{i}.\) Suppose the first apparition of \(\rho\) is pulled back to a single vertex (i.e., no hidden cluster at this location): this can be any element of \(s(k);\) for each of the next \(t-2\) copies, if it is pulled to a cluster, then the right-most vertex can be any element in \(s(k),\) whereas the left-most is predetermined (the edge adjacent to it preceding the cluster is unmarked); else, it is pulled to a single vertex, and there is just one possibility (the left-most vertex is still unmarked and the right-most has to coincide with it); finally, the last copy of \(\rho\) is also fully determined since both edges are unmarked. Hence there are \(q+1\) vertices to choose, from which the result follows. If \((s_0,\rho,s_1)\) is pulled to a cluster, then a similar reasoning gives two vertices to pick for it (both are marked), and \(q-1\) or \(q-2\) more (depending on whether the last copy of \(\rho\) is pulled back to a cluster or a single vertex). 
\par
\underline{Case 2:} \(\rho\) is the first vertex of \(\mathbf{i}_c\) with apparitions \((\rho,s_1), (s_2,\rho,s_3), \hspace{0.05cm} ... \hspace{0.05cm}, (s_{2t-2},\rho,s_{2t-1}),(s_{2t},\rho);\) because \(\newline (s_2,\rho),\hspace{0.05cm}  ... \hspace{0.05cm} ,(s_{2t-2},\rho),(s_{2t},\rho)\) are unmarked, \((\rho,s_3),\hspace{0.05cm}  ... \hspace{0.05cm} ,(\rho,s_{2t-1}),(s_{2t},\rho)\) are marked. The rationale is analogous to the one employed in case \(1,\) concluding (\ref{longsum}).
\par
Let \(l=px,q=py,t=pz,\) and employ again (\ref{stirling}), which renders with Lemma~\ref{lemma9},
\[\lim_{p \to \infty}{(s_2(\theta,p))^{1/p}}=f^2(\theta).\]
This concludes the proof of (\ref{v22}): (\ref{v222}) can be justified in a similar vein (in light of (\ref{longsum})).

\begin{lemma}\label{lemma9}
Let \(D=\{(x,y,z) \in [0,1]^3: y \leq z \leq x, y \leq 2-2x\},\) and \(h:D \to \mathbb{R},\)
\[h(x,y,z)=\theta^{2-2x} \cdot \frac{z^z}{y^y \cdot (z-y)^{z-y}} \cdot \frac{(2-2x)^{2-2x}}{y^y \cdot (2-2x-y)^{2-2x-y}} \cdot \frac{(2x-z)^{2x-z}}{x^x \cdot (x-z)^{x-z}},\]
where \(0^0:=\lim_{x \to 0+}{x^x}=1.\) Then
\[\sup_{(x,y,z) \in D}{h(x,y,z)}=f^2(\theta).\]
\end{lemma}

\begin{proof}
Note
\[\frac{1}{h} \cdot \frac{\partial{h}}{\partial{y}}=\log{\frac{(z-y)(2-2x-y)}{y^2}}, \hspace{0.5cm} \frac{1}{h} \cdot \frac{\partial{h}}{\partial{z}}=\log{\frac{z(x-z)}{(z-y)(2x-z)}};\]
in particular,
\[\frac{\partial{h}}{\partial{y}}=0 \Longleftrightarrow y=\frac{z(2-2x)}{z+2-2x}, \hspace{1cm} \frac{\partial{h}}{\partial{z}}=0 \Longleftrightarrow z=\frac{2xy}{x+y}.\]
\par
By a slight abuse of notation, let \((x,y,z) \in \arg{\sup{h}}.\) Consider \(z \in [y,x]:\) if \(y=0,\) then 
\[h(x,y,z)=\theta^{2-2x} \cdot \frac{(2x-z)^{2x-z}}{x^x(x-z)^{x-z}}:=g_1(x,z)\] 
with \(\frac{1}{g_1} \cdot \frac{\partial{g_1}}{\partial{z}}=\log{\frac{x-z}{2x-z}}<0\) for \(z \in (0,x),\) and \(\sup_{x \in [0,1]}{g_1(x,0)}=\max{(\theta^2,4)}.\) Else, \(x,y>0,\) 
\[\frac{1}{h} \cdot \frac{\partial{h}}{\partial{z}}=\log{\frac{z(x-z)}{(z-y)(2x-z)}}=\log{(1+\frac{2xy-(x+y)z}{(z-y)(2x-z)})}\]
and \(\frac{2xy}{x+y} \in [y,x]\) entail \(z=\frac{2xy}{x+y}.\) What remains is \(y \leq \min{(x, 2-2x)},\) and 
\[g_2(x,y)=\theta^{2-2x} \cdot (\frac{z-y}{y})^y \cdot \frac{(2-2x)^{2-2x}}{y^y \cdot (2-2x-y)^{2-2x-y}} \cdot \frac{(2x-z)^{2x}}{x^x \cdot (x-z)^{x}}\]
because \(\frac{\partial{g_1}}{\partial{z}}=0;\) employing \(\frac{z-y}{y}=\frac{x-y}{x+y},2-\frac{z}{x}=\frac{2x}{x+y}, 1-\frac{z}{x}=\frac{x-y}{x+y},\)
\[g_2(x,y)=\theta^{2-2x} \cdot \frac{(2-2x)^{2-2x}}{y^y \cdot (2-2x-y)^{2-2x-y}} \cdot \frac{(2x)^{2x}}{(x-y)^{x-y}(x+y)^{x+y}}.\]
Since
\[\frac{1}{g_2} \cdot \frac{\partial{g_2}}{\partial{y}}=\log{\frac{(x-y)(2-2x-y)}{y(x+y)}}=\log{(1+\frac{2x(1-x)-2y}{y(x+y)})},\] 
the maximum is attained at \(y=x(1-x) \leq \min{(x,1-x)} \leq \min{(x,2-2x)}.\) \(\frac{\partial{g_2}}{\partial{y}}=0\) renders
\[q(x)=\theta^{2-2x} \cdot \frac{(2-2x)^{2-2x}}{(2-2x-y)^{2-2x}} \cdot \frac{(2x)^{2x}}{(x-y)^{x}(x+y)^{x}},\]
which together with \(1-\frac{y}{2-2x}=\frac{2-x}{2},1-\frac{y^2}{x^2}=1-(1-x)^2=x(2-x)\) gives
\[q(x)=\theta^{2-2x} \cdot \frac{4}{x^x(2-x)^{2-x}}.\]
Finally, 
\(\frac{q'(x)}{q(x)}=\log{\frac{2-x}{\theta^2x}}.\) The critical point is \(x_2=\frac{2}{\theta^2+1}\) if \(\theta \geq 1, q(\frac{2}{\theta^2+1})=\frac{4\theta^2}{(2-x_2)^2}=(\theta+\frac{1}{\theta})^2,\)
yielding the supremum of \(h\) is
\[\max{(4,(\theta+\frac{1}{\theta})^2\chi_{\theta \geq 1})}=f^2(\theta)\]
as \(q(0)=\theta^2,q(1)=4.\) This completes the proof of the lemma.
\end{proof}


\end{document}